\documentclass[12pt]{article}
\usepackage[utf8]{inputenc}
\usepackage{amsthm}
\usepackage{amsmath}
\usepackage{bbm}
\usepackage{amsfonts}
\usepackage{amssymb}
\usepackage[left=2.5cm,right=2.5cm,top=2cm,bottom=2cm]{geometry}
\usepackage[linktocpage=true]{hyperref}

\hypersetup{
  colorlinks   = true, %Colours links instead of ugly boxes
  urlcolor     = blue, %Colour for external hyperlinks
  linkcolor    = blue, %Colour of internal links
  citecolor   = red %Colour of citations
}

\usepackage{tabto}
\TabPositions{2 cm, 4 cm, 6 cm, 8 cm}

\usepackage{tikz-cd}

\usepackage{tikz}
\usetikzlibrary{arrows, automata, positioning}

\usepackage{tocloft}
\usepackage{appendix}

\usepackage{enumerate}

\usepackage{hyperref}
 
\newtheorem{theorem}{Theorem}[section]
\newtheorem{corollary}[theorem]{Corollary}
\newtheorem{lemma}[theorem]{Lemma}

\newtheorem{proposition}[theorem]{Proposition}

\newtheorem*{lemma*}{Lemma}
\newtheorem*{proposition*}{Proposition}
\newtheorem*{theorem*}{Theorem}
\newtheorem*{corollary*}{Corollary}
\newtheorem*{claim*}{Claim}

\theoremstyle{definition}
\newtheorem{definition}[theorem]{Definition}
\newtheorem{example}[theorem]{Example}
\newtheorem*{definition*}{Definition}
\newtheorem{remark}[theorem]{Remark}

\newtheorem{question}[theorem]{Question}

\newcommand{\im}{\operatorname{im}}

\newcommand{\Homeo}{\operatorname{Homeo}}
\newcommand{\Ab}{\operatorname{Ab}}

\newcommand{\res}{\operatorname{res}}
\newcommand{\trans}{\operatorname{trans}}
\newcommand{\id}{\operatorname{id}}

\newcommand{\N}{\mathbb{N}}
\newcommand{\R}{\mathbb{R}}

\newcommand{\HH}{\operatorname{H}}
\newcommand{\Ha}{\operatorname{H}_a}

\newcommand{\defe}{\operatorname{def}}
\newcommand{\dist}{\operatorname{dist}}
\newcommand{\Uf}{\mathfrak{U}}
\newcommand{\uf}{\mathfrak{u}}

\newcommand{\hR}{{}^{*}\mathbb{R}}
\newcommand{\hRb}{{}^{*}\mathbb{R}_b}
\newcommand{\hRinf}{{}^{*}\mathbb{R}_{inf}}
\newcommand{\hGamma}{{}^{*}\Gamma}
\newcommand{\hLambda}{{}^{*}\Lambda}
\newcommand{\hS}{{}^{*}S}

\newcommand{\V}{\mathcal{V}}
\newcommand{\Vb}{\mathcal{V}_b}
\newcommand{\Vinf}{\mathcal{V}_{inf}}
\newcommand{\Vtilde}{\tilde{\mathcal{V}}}

\newcommand{\W}{\mathcal{W}}

\newcommand{\Wtilde}{\tilde{\mathcal{W}}}

\newcommand{\LL}{\mathcal{L}^\infty}
\newcommand{\LLb}{\mathcal{L}^\infty_b}
\newcommand{\LLinf}{\mathcal{L}^\infty_{inf}}
\newcommand{\LLtilde}{\tilde{\mathcal{L}}^\infty}

\newcommand{\eps}{\varepsilon}

\begin{document}

\title{Ulam stability of lamplighters and Thompson groups}
\author{Francesco Fournier-Facio and Bharatram Rangarajan}
\date{\today}
\maketitle

\begin{abstract}
We show that a large family of groups is uniformly stable relative to unitary groups equipped with submultiplicative norms, such as the operator, Frobenius, and Schatten $p$-norms. These include lamplighters $\Gamma \wr \Lambda$ where $\Lambda$ is infinite and amenable, as well as several groups of dynamical origin such as the classical Thompson groups $F, F', T$ and $V$. We prove this by means of vanishing results in asymptotic cohomology, a theory introduced by the second author, Glebsky, Lubotzky and Monod, which is suitable for studying uniform stability. Along the way, we prove some foundational results in asymptotic cohomology, and use them to prove some hereditary features of Ulam stability. We further discuss metric approximation properties of such groups, taking values in unitary or symmetric groups.
\end{abstract}

\section{Introduction}

Let $\Gamma$ be a countable discrete group, and let $\Uf$ be a family of finite-dimensional unitary groups. The problem of stability asks whether every almost-homomorphism $\Gamma \to U \in \Uf$ is close to a homomorphism. To formalize this we need to choose a norm, and a way to interpret these approximate notions. We focus on the classical setting of \emph{uniform} defects and distances, with respect to \emph{submultiplicative} norms.

Let $\Uf := \{ (U(k), \| \cdot \|) \}$ be a family of finite-dimensional unitary groups equipped with bi-invariant submultiplicative norms $\| \cdot \|$ (we allow $U(k)$ to appear multiple times with different norms). For instance $\| \cdot \|$ could be the operator norm - the most classical case - or more generally a Schatten $p$-norm.
Given a map $\phi : \Gamma \to U(k)$, we define its \emph{defect} to be
$$\defe(\phi) := \sup\limits_{g, h \in \Gamma} \| \phi(gh) - \phi(g) \phi(h)\|.$$
Given another map $\psi : \Gamma \to U(k)$, we define the \emph{distance} between them to be
$$\dist(\phi, \psi) := \sup\limits_{g \in \Gamma} \| \phi(g) - \psi(g) \|.$$

\begin{definition}
A \emph{uniform asymptotic homomorphism} is a sequence of maps $\phi_n : \Gamma \to U(k_n)$ such that $\defe(\phi_n) \to 0$. We denote this simply by $\phi : \Gamma \to \Uf$.
We say that $\phi, \psi : \Gamma \to \Uf$ are \emph{uniformly asymptotically close} if they have the same range degrees and $\dist(\phi_n, \psi_n) \to 0$.

The group $\Gamma$ is \emph{uniformly $\Uf$-stable} if every uniform asymptotic homomorphism is uniformly asymptotically close to a sequence of homomorphisms.
\end{definition}

We can also talk quantitatively about stability, by asking how close a homomorphism we can choose, in terms of the defect. This leads to the notion of \emph{stability with a linear estimate}, which will be relevant for us and which we define precisely in Section \ref{ss:preli:stab}. \\

Early mentions of similar problems can be found in the works of von Neumann \cite{vN} and Turing \cite{turing}. In \cite[Chapter 6]{ulam} Ulam discussed more general versions of stability, which has since inspired a large body of work. Uniform $\Uf$-stability has been studied mostly when $\Uf$ is the family of unitary groups equipped with the operator norm, for which the notion is typically referred to as \emph{Ulam stability}. In this contest, Kazhdan proved stability of amenable groups \cite{kazhdan}, while Burger, Ozawa and Thom proved stability of certain special linear groups over $S$-integers, and instability of groups admitting non-trivial quasimorphisms \cite{BOT}.

More recently, the second author, Glebsky, Lubotzky and Monod proved Ulam stability of certain lattices in higher rank Lie groups, with respect to arbitrary submultiplicative norms \cite{mainref}. For the proof, they introduce a new cohomology theory, called \emph{asymptotic cohomology}, and prove that stability is implied by the vanishing of certain asymptotic cohomology classes $\alpha \in \HH^2_a(\Gamma, \V)$. We refer the reader to Section \ref{ss:preli:asy} for the relevant definitions. \\

The goal of this paper is to further the understanding of asymptotic cohomology, and apply this to prove new stability results. The main one is the stability of the classical Thompson groups:

\begin{theorem}[Section \ref{s:thompson}]
\label{intro:thm:F}

Thompson's groups $F, F', T$ and $V$ are uniformly $\Uf$-stable, with a linear estimate.
\end{theorem}

As remarked by Arzhantseva and P\u{a}unescu \cite[Open problem]{arzhantsevapaunescu}, the analogous statement for pointwise stability in permutation of $F$ would imply that $F$ is not sofic, thus proving at once the existence of a non-sofic group and the non-amenability of $F$: two of the most remarkable open problems in modern group theory. We will discuss these problems and their relation to our results in Section \ref{s:last}. \\

Theorem \ref{intro:thm:F} for $F$ and $F'$ will follow from a stability result for certain \emph{lamplighters}. Given groups $\Gamma, \Lambda$, the corresponding lamplighter (or \emph{restricted wreath product}) is the group $\Gamma \wr \Lambda = (\oplus_\Lambda \Gamma)\rtimes \Lambda$, where $\Lambda$ acts by shifting the coordinates.

\begin{theorem}
\label{intro:thm:lamplighters}

Let $\Gamma, \Lambda$ be two countable groups, where $\Lambda$ is infinite and amenable. Then $\Gamma \wr \Lambda$ is uniformly $\Uf$-stable, with a linear estimate.
\end{theorem}

By itself, Theorem \ref{intro:thm:lamplighters} provides a plethora of examples of uniformly $\Uf$-stable groups, to a degree of flexibility that was not previously available. For instance, using classical embedding results \cite{embedding} it immediately implies the following:

\begin{corollary}
\label{intro:cor:many}

Every countable group embeds into a $3$-generated group which is uniformly $\Uf$-stable, with a linear estimate.
\end{corollary}

In particular, this gives a proof that there exist uncountably many finitely generated uniformly $\Uf$-stable groups, a fact which could also be obtained by applying Kazhdan's Theorem \cite{kazhdan} to an infinite family of finitely generated amenable groups, such as the one constructed by B. H. Neumann \cite{BH}. \\

In order to obtain stability of $F$ and $F'$ from Theorem \ref{intro:thm:lamplighters}, we exploit \emph{coamenability}. Recall that a subgroup $\Lambda \leq \Gamma$ is coamenable if the coset space $\Gamma / \Lambda$ admits a $\Gamma$-invariant mean. It is well known that $F'$ and $F$ contain a coamenable lamplighter $F \wr \mathbb{Z}$. Therefore the stability of $F$ and $F'$ (Corollary \ref{cor:F}) follows from Theorem \ref{intro:thm:lamplighters}, and the following result:

\begin{proposition}
\label{intro:prop:coamenable}

Let $\Lambda \leq \Gamma$ be coamenable. If $\Lambda$ is uniformly $\Uf$-stable with a linear estimate, then so is $\Gamma$.
\end{proposition}

This can be seen as a relative version of the celebrated result of Kazhdan, stating that amenable groups are uniformly $\Uf$-stable \cite{kazhdan}. To complete the picture, we also prove another relative version of Kazhdan's Theorem, which is sort of dual to Proposition \ref{intro:prop:coamenable}:

\begin{proposition}
\label{intro:prop:mapping}

Let $N \leq \Gamma$ be an amenable normal subgroup. If $\Gamma$ is uniformly $\Uf$-stable with a linear estimate, then so is $\Gamma / N$.
\end{proposition}

The fact that Theorem \ref{intro:thm:F} follows from Theorem \ref{intro:thm:lamplighters} and Proposition \ref{intro:prop:coamenable} is not special to Thompson's group $F$: this phenomenon is typical of several groups of piecewise linear and piecewise projective homeomorphisms, which enjoy some kind of self-similarity properties (Theorem \ref{thm:selfsimilar} and Corollary \ref{cor:bsupp}).
Stability of $T$ and $V$ then follow from these results, together with a bounded generation argument analogous to the one from \cite{BOT} (Corollaries \ref{cor:T} and \ref{cor:V}). \\

As we mentioned above, the tool underlying the proofs of Theorem \ref{intro:thm:lamplighters} and Proposition \ref{intro:prop:coamenable} is asymptotic cohomology, in particular the vanishing of certain classes in degree $2$. In this framework, Theorem \ref{intro:thm:lamplighters} takes the following form:

\begin{theorem}
\label{intro:thm:lamplighters:ac}

Let $\Gamma, \Lambda$ be two countable groups, where $\Lambda$ is infinite and amenable. Then $\Ha^n(\Gamma \wr \Lambda, \V) = 0$ for all $n \geq 1$ and all finitary dual asymptotic Banach $\hGamma$-modules $\V$.
\end{theorem}

Here the word \emph{finitary} refers to the fact that these modules arise from stability problems with respect to finite-dimensional unitary representations. This hypothesis is crucial: see Remark \ref{rem:sharp:lamplighters}.
Propositions \ref{intro:prop:coamenable} and \ref{intro:prop:mapping} also follow from results in asymptotic cohomology, that this time does not need the finitary assumption:

\begin{proposition}
\label{intro:prop:coamenable:ac}

Let $\Lambda \leq \Gamma$ be coamenable. Then the restriction map $\Ha^n(\Gamma, \V) \to \Ha^n(\Lambda, \V)$ is injective, for all $n \geq 0$ and all dual asymptotic Banach $\hGamma$-modules $\V$.
\end{proposition}

\begin{proposition}
\label{intro:prop:mapping:ac}

Let $N \leq \Gamma$ be an amenable normal subgroup. Then the pullback map $\Ha^n(\Gamma/N, \V) \to \Ha^n(\Gamma, \V)$ is an isomorphism, for all $n \geq 0$ and all dual asymptotic Banach ${}^{*} (\Gamma/N)$-modules $\V$.
\end{proposition}

Despite the lack of a general theorem connecting the two theories, asymptotic cohomology seems to be closely connected to \emph{bounded cohomology}, a well-established cohomology theory \cite{johnson, gromov, ivanov, monod:book, frigerio} that has become a fundamental tool in rigidity theory. The vanishing result for asymptotic cohomology of lattices leading to stability \cite{mainref} follows closely the vanishing result for bounded cohomology of high-rank lattices \cite{lattices, rigidity, monodshalom}.
Similarly, our proofs of Theorem \ref{intro:thm:lamplighters:ac} and Propositions \ref{intro:prop:coamenable:ac} and \ref{intro:prop:mapping:ac} follow closely the corresponding bounded-cohomological results: for Theorem \ref{intro:thm:lamplighters:ac} this was recently proven by Monod \cite{monod:lamplighters}, while for Proposition \ref{intro:prop:coamenable:ac} this is a foundational result in bounded cohomology \cite[8.6]{monod:book} (see also \cite{coamenable}), and Proposition \ref{intro:prop:mapping:ac} is an analogue of Gromov's Mapping Theorem \cite{gromov}. Note that the bounded cohomology of $T$ and $V$ has also been recently computed \cite{binate, monodnariman, konstantin}, but only with trivial real coefficients, and our proofs are of a different nature.

We thus hope that the steps we undertake to prove our main results will be useful to produce more computations in asymptotic cohomology, and therefore more examples of uniformly $\Uf$-stable, and in particular Ulam stable, groups. \\

Our results have applications to the study of approximating properties of groups. While questions on pointwise approximation, such as soficity, hyperlinearity, and matricial finiteness, are in some sense disjoint from the content of this paper, our stability results imply that some of the groups considered in this paper are not \emph{uniformly approximable} with respect to the relevant families $\Uf$ (Corollary \ref{cor:F:uniapprox}). We are also able to treat the case of symmetric groups endowed with the Hamming distance, by a more direct argument (Proposition \ref{prop:sym}). \\

We end this introduction by proposing a question. There is a notion of \emph{strong Ulam stability}, where the approximations take values in unitary groups of possibly infinite-dimensional Hilbert spaces, with the operator norm. It is a well-known open question whether strong Ulam stability coincides with amenability. In this direction it is known that strong Ulam stable groups have no non-abelian free subgroups \cite[Theorem 1.2]{BOT}, but there exist groups without non-abelian free subgroups that are not strong Ulam stable \cite{alpeev}.

On the other hand, our results also prove uniform $\Uf$-stability stability of the piecewise projective groups of Monod \cite{PP1} and Lodha--Moore \cite{PP3}, which are nonamenable and without free subgroups (see Section \ref{s:PLPP}). Therefore we ask the following:

\begin{question}\label{q}
Let $\Gamma$ be a countable group without non-abelian free subgroups. Is $\Gamma$ uniformly $\Uf$-stable (with a linear estimate)? Or at least Ulam stable?

In particular, are all countable torsion groups Ulam stable?
\end{question}

In other words: if $\Gamma$ is not Ulam stable, must $\Gamma$ contain a non-abelian free subgroup? To our knowledge it is not even known if groups admitting non-trivial quasimorphisms must contain non-abelian free subgroups: see \cite{manning} and \cite{laws} for partial results in this direction. \\

\textbf{Conventions:} All groups are assumed to be discrete and countable. The set of natural numbers $\N$ starts at $0$. A non-principal ultrafilter $\omega$ on $\N$ is fixed for the rest of the paper. \\

\textbf{Outline:} We start in Section \ref{s:preli} by reviewing the framework of asymptotic cohomology and its applications to stability, as developed in \cite{mainref}.
In Section \ref{s:hereditary} we discuss hereditary properties for Ulam stability, and prove Propositions \ref{intro:prop:coamenable} and \ref{intro:prop:mapping}. We then move to lamplighters and prove Theorem \ref{intro:thm:lamplighters} in Section \ref{s:lamplighters}, then to Thompson groups proving Theorem \ref{intro:thm:F} in Section \ref{s:thompson}. In Section \ref{s:sharp} we provide examples showing that some of our results and some of the results from \cite{mainref} are sharp, and conclude in Section \ref{s:last} by discussing applications to the study of metric approximations of groups. \\

\textbf{Acknowledgements:} The authors are indebted to Alon Dogon, Lev Glebsky, Alexander Lubotzky and Nicolas Monod for useful conversations.

\section{Uniform stability and asymptotic cohomology}
\label{s:preli}

In this section, we shall briefly summarize the notion of defect diminishing that allows us to formulate the stability problem as a problem of lifting of homomorphisms with abelian kernel, which in turn motivates the connection to second cohomology. For a more detailed description, refer to Section $2$ in \cite{mainref}.

\subsection{Uniform stability and defect diminishing}
\label{ss:preli:stab}
We begin by reviewing some basic notions of ultraproducts and non-standard analysis, before formulating the stability problem as a homomorphism lifting problem. For this, it is convenient to describe a uniform asymptotic homomorphism (which is a sequence of maps) as one map of ultraproducts. This in turn allows us to perform a soft analysis to obtain a (true) homomorphism to a quotient group. 
Recall that $\omega$ is a fixed non-principal ultrafilter on $\N$. The \emph{algebraic ultraproduct} $\prod_{\omega}X_n$ of an indexed collection $\{X_n\}_{n \in \N}$ of sets is defined to be $\prod_{\omega} X_n := \prod_{n \in \N} X_n / \sim$ where for $\{x_n\}_{n \in \N}, \{y_n\}_{n \in \N} \in \prod_{n \in \N}X_n$, we define $\{x_n\}_{n \in \N} \sim \{y_n\}_{n \in \N}$ if $\{n \space : \space x_n=y_n\} \in \omega$. Ultraproducts can be made to inherit algebraic structures of their building blocks. For instance, for a group $\Gamma$, the ultraproduct $\prod_{\omega}\Gamma$, called the \emph{ultrapower} and denoted $\hGamma$, is itself a group. Another important example we will use is the field of hyperreals $\hR$, the ultrapower of $\mathbb{R}$.\\

Objects (sets, functions, etc.) that arise as ultraproducts of standard objects are referred to as \emph{internal}. Important examples of non-internal objects are the subsets $\hR_b$ of \emph{bounded} hyperreals, consisting of elements $\{x_n\}_{\omega} \in \hR$ for which there exists $S \in \omega$ and $C \in \R_{\geq 0}$ such that $\lvert x_n \rvert \leq C$ for every $n \in S$, and the subset $\hR_{inf}$ of \emph{infinitesimals}, consisting of elements $\{x_n\}_{\omega} \in \hR$ such that for every real $\eps > 0$, there exists $S \in \omega$ such that $\lvert x_n \rvert < \eps$ for every $n \in S$.

For $x,y \in \hR$, write $x=O_{\omega}(y)$ if $x/y \in \hR_b$, and write $x=o_{\omega}(y)$ if $x/y \in \hR_{inf}$. In particular, $x \in \hR_b$ is equivalent to $x=O_{\omega}(1)$ while $\eps \in \hR_{inf}$ is equivalent to $\eps = o_{\omega}(1)$.
The subset $\hR_b$ forms a valuation ring with $\hR_{inf}$ being the unique maximal ideal, with quotient $\hR_b/\hR_{inf} \cong \R$. The quotient map $st:\hR_b \to \R$ is known as the \emph{standard part} map or \emph{limit along the ultrafilter $\omega$}. The previous construction can also be replicated for Banach spaces. Let $\{W_n\}_{n \in \N}$ be a family of Banach spaces. Then $\mathcal{W}=\prod_{\omega}W_n$ can be given the structure of a $\hR$-vector space. In fact, it also comes equipped with a $\hR$-valued norm, allowing us to define the external subsets $\mathcal{W}_b$ and $\mathcal{W}_{inf}$. The quotient $\Wtilde := \mathcal{W}_b/\mathcal{W}_{inf}$ is a real Banach space.\\

Given a uniform asymptotic homomorphism $\{\phi_n:\Gamma \to U(k_n)\}_{n \in \N}$ with $\defe(\phi_n) =: \eps_n \to 0$, construct the internal map $\phi : \hGamma \to \prod_{\omega}U(k_n)$ where $\phi := \prod_{\omega}\phi_n$, with (hyperreal) defect $\eps := \{\eps_n\}_{\omega} \in \hR_{inf}$. Then the question of uniform stability with a linear estimate can be rephrased as asking whether there exists an internal homomorphism $\psi:\hGamma \to \prod_{\omega}U(k_n)$ such that their (hyperreal) distance satisfies $\dist(\phi,\psi) := \{\dist(\phi_n,\psi_n)\}_{\omega} = O_{\omega}(\eps)$. 

The advantage of rephrasing the question in terms of internal maps is that an internal map $\phi:\hGamma \to \prod_{\omega}U(k_n)$ with defect $\eps \in \hR_{inf}$ induces a true homomorphism $\tilde{\phi}:\hGamma \to \prod_{\omega}U(k_n)/B(\eps)$ where $B(\eps)$ is the (external) normal subgroup of $\prod_{\omega}U(k_n)$ comprising elements that are at a distance $O_{\omega}(\eps)$ from the identity. In particular, the question of uniform stability with a linear estimate can equivalently be rephrased as asking whether given such an internal map $\phi$, can the homomorphism $\tilde{\phi}:\hGamma \to \prod_{\omega}U(k_n)/B(\eps)$ be lifted to an internal homomorphism $\psi:\hGamma \to \prod_{\omega}U(k_n)$.\\

Reinterpreting uniform stability with a linear estimate as a homomorphism lifting problem motivates a cohomological approach to capturing the obstruction. However, the obstacle here is that the kernel $B(\eps)$ of the lifting problem is not abelian. This can be handled by lifting in smaller steps so that each step involves an abelian kernel. Define a normal subgroup $I(\eps)$ of $B(\eps)$ comprising elements that are at a distance of $o_{\omega}(\eps)$ from the identity. Then we can attempt to lift $\tilde{\phi}:\hGamma \to \prod_{\omega}U(k_n)/B(\eps)$ to an internal map $\psi:\hGamma \to \prod_{\omega}U(k_n)$ that is a homomorphism modulo $I(\eps)$. The problem is simpler from the cohomological point of view: since the norms are submultiplicative, the kernel $B(\eps)/I(\eps)$ of this lifting problem is abelian. The group $\Gamma$ is said to have the \emph{defect diminishing property} with respect to $\Uf$ if such a lift exists; more explicitly, $\Gamma$ has the defect diminishing property if for every uniform asymptotic homomorphism $\phi : \Gamma \to \Uf$ there exists a uniform asymptotic homomorphism $\psi$ with the same range such that $\dist(\phi, \psi) = O_{\omega}(\defe(\phi))$ and $\defe(\psi) = o_{\omega}(\defe(\phi))$.

\begin{theorem}[{\cite[Theorem 2.3.11]{mainref}}]
\label{thm:defdim}

$\Gamma$ has the defect diminishing property with respect to $\Uf$ if and only if $\Gamma$ is uniformly $\Uf$-stable with a linear estimate.
\end{theorem}

The obstruction to such a homomorphism lifting, with an abelian kernel $B(\eps)/I(\eps)$, can be carefully modeled using a cohomology $\HH_a^{\bullet}(\Gamma, \W)$ so that $\HH_a^2(\Gamma,\W)=0$ implies the defect diminishing property (and consequently, uniform stability with a linear estimate). Here $\mathcal{W}=\prod_{\omega}\uf(k_n)$ is an internal Lie algebra of $\prod_{\omega}U(k_n)$ equipped with an asymptotic action of the ultrapower $\hGamma$ constructed from the uniform asymptotic homomorphism $\phi$ that we start out with. The logarithm of the defect map
$$\hGamma \times \hGamma \to \prod_{\omega}U(k_n) : (g_1,g_2) \mapsto \phi(g_1)\phi(g_2)\phi(g_1g_2)^{-1}$$ would correspond to an asymptotic $2$-cocycle in $\HH_a^2(\Gamma,\W)$. Such a cocycle is a coboundary in this setting (that is, it represents the zero class in $\HH_a^2(\Gamma, \W)$), if and only if the defect diminishing property holds for the asymptotic homomorphism $\phi$.

\subsection{Asymptotic cohomology}
\label{ss:preli:asy}

The reduction to a lifting problem with abelian kernel motivates a cohomology theory of $\Gamma$ with coefficients in the internal Lie algebra $\mathcal{W}=\prod_{\omega}\uf(k_n)$ of $\prod_{\omega}U(k_n)$, equipped with an asymptotic conjugation action of $\Gamma$. In this section we review the formal definition of this cohomology, and state some results from \cite{mainref} that we shall need to work with it.\\
 
Let $(V_n)_{n \geq 1}$ be a sequence of Banach spaces, and let $\V := \prod\limits_\omega V_n$ be their algebraic ultraproduct: we refer to such $\V$ as an \emph{internal} Banach space. For $v \in \V$ we denote by $\| v \|$ the hyperreal $(\| v_n \|)_\omega \in \hR$. We then denote by
$$\Vb := \{ v \in \V : \| v \| \in \hRb \}; \qquad \Vinf := \{ v \in \V : \| v \| \in \hRinf \}.$$
Then the quotient $\Vtilde := \Vb/\Vinf$ is a real Banach space, whose norm is induced by the ultralimit of $\| \cdot \|$ on $\Vb$. For each $V_n$ denote by $V_n^\#$ its continuous dual, and let $\V^\#$ be the corresponding algebraic ultraproduct. The pairing $\langle \cdot, \cdot \rangle_n : V_n^\# \times V_n \to \mathbb{R}$ induces a pairing $\V^\# \times \V \to \hR$ which descends to $\Vtilde^\# \times \Vtilde \to \mathbb{R}$. We call $\V^\#$ the \emph{internal dual} of $\V$.

Now let $\Gamma$ be a countable discrete group, and let $\pi : \hGamma \times \V \to \V$ be an internal map which preserves $\| \cdot \|$ and induces an isometric linear action $\tilde{\pi} : \hGamma \times \Vtilde \to \Vtilde$ of $\hGamma$. Such a map $\pi$ is referred to as an \emph{asymptotic $\hGamma$-action on $\V$}. We then call $(\pi, \V)$, or $\V$, if $\pi$ is understood from context, an \emph{asymptotic Banach $\hGamma$-module}. Given an internal Banach $\hGamma$-module $(\pi, \V)$, the contragradient on each coordinate induces an internal map $\pi^\# : \hGamma \times \V^\# \to \V^\#$ making $(\pi^\#, \V^\#)$ into an asymptotic Banach $\hGamma$-module. We call a module $\V$ a \emph{dual asymptotic Banach $\hGamma$-module} if $\V$ is the dual of some asymptotic $\hGamma$-module denoted $\V^{\flat}$. We decorate these definitions with the adjective \emph{finitary} if each $V_n$ is finite-dimensional. \\

Now for each $m \geq 0$ define the internal Banach space $\LL((\hGamma)^m, \V) := \prod\limits_{\omega} \ell^\infty(\Gamma^m, V_n)$ (note that $m$ is fixed and $n$ runs through the natural numbers with respect to the ultrafilter $\omega$). Similarly to before, for $f \in \LL((\hGamma)^m, \V)$ we denote $\| f \| := (\| f_n \|)_\omega \in \hR$ and
$$\LLb((\hGamma)^m, \V) := \{ f \in \LL((\hGamma)^m, \V) : \| f \| \in \hRb \};$$
$$\LLinf((\hGamma)^m, \V) := \{ f \in \LL((\hGamma)^m, \V) : \| f \| \in \hRinf \}.$$
Given an asymptotic $\hGamma$-action $\pi$ on $\V$, we can construct a natural asymptotic $\hGamma$-action $\rho^m : \hGamma \times \LL((\hGamma)^m, \V) \to \LL((\hGamma)^m, \V)$ given by
\begin{equation}\label{rho}
(\rho^m(g)(f))(g_1,g_2,\dots,g_m) := \pi_G(g)f(g^{-1}g_1,\dots,g^{-1}g_m)
\end{equation}

Then the quotient
$$\LLtilde((\hGamma)^m, \V) := \LLb((\hGamma)^m, \V) / \LLinf((\hGamma)^m, \V)$$
is again a real Banach space equipped with an isometric $\hGamma$-action induced coordinate-wise by $\rho^m$, which defines the invariant subspaces $\LLtilde((\hGamma)^m, \V)^{\hGamma}$. \\

Now define the internal coboundary map
$$d^m : \LL((\hGamma)^m, \V) \to \LL((\hGamma)^{m+1}, \V);$$
\begin{equation}\label{coboundary}
    d^m(f)(g_0, \ldots, g_m) := \sum\limits_{j = 0}^m (-1)^j f(g_0, \ldots, \hat{g_j}, \ldots, g_m),
\end{equation}
which descends to coboundary maps
$$\tilde{d^m}: \LLtilde((\hGamma)^m, \V) \to \LLtilde((\hGamma)^{m+1}, \V).$$
Since $\tilde{d^m}$ is $\hGamma$-equivariant, it defines the cochain complex:
$$0 \xrightarrow{\tilde{d^0}} \LLtilde(\hGamma, \V)^{\hGamma} \xrightarrow{\tilde{d^1}} \LLtilde((\hGamma)^2, \V)^{\hGamma} \xrightarrow{\tilde{d^2}} \LLtilde((\hGamma)^3, \V)^{\hGamma} \xrightarrow{\tilde{d^3}} \cdots $$

\begin{definition}[{\cite[Definition 4.2.2]{mainref}}]
The $m$-th \emph{asymptotic cohomology} of $\Gamma$ with coefficients in $\V$ is
$$\Ha^m(\Gamma, \V) := \ker(\tilde{d^{m+1}}) / \im(\tilde{d^m}).$$
\end{definition}

Other resolutions may also be used to compute asymptotic cohomology. Recall (\cite[5.3.2]{monod:book}) that a regular $\Gamma$-space $S$ is said to be a \emph{Zimmer-amenable $\Gamma$-space} if there exists a $\Gamma$-equivariant conditional expectation $\mathfrak{m}:L^{\infty}(\Gamma \times S) \to L^{\infty}(S)$. Let $S$ be a regular $\Gamma$-space with a Zimmer-amenable action of $\Gamma$, and let $\LL((\hS)^m, \V) := \prod\limits_{\omega} L^\infty_{w*}(S^m, V_n)$ (where $L^\infty_{w*}(S^m, V_n)$ is the space of bounded weak-$*$ measurable function classes from $S^m$ to $V_n$). Again, the asymptotic $\hGamma$-action on $\V$ gives rise to a natural asymptotic $\hGamma$-action on $\LL((\hGamma)^m, \V)$ as in (\ref{rho}), making $\LL((\hS)^m, \V)$ an asymptotic Banach $\hGamma$-module. The coboundary maps too can be defined just as in (\ref{coboundary}), to construct the cochain complex, and we have:

\begin{theorem}[{\cite[Theorem 4.3.3]{mainref}}]
\label{thm:zimmer}

Let $S$ be a Zimmer-amenable $\Gamma$-space, and $\V$ be a dual asymptotic Banach $\hGamma$-module. Then $\HH_a^{\bullet}(\Gamma,\V)$ can be computed as the asymptotic cohomology of the cochain complex
$$0 \xrightarrow{\tilde{d^0}} \LLtilde(\hS, \V)^{\hGamma} \xrightarrow{\tilde{d^1}} \LLtilde((\hS)^2, \V)^{\hGamma} \xrightarrow{\tilde{d^2}} \LLtilde((\hS)^3, \V)^{\hGamma} \xrightarrow{\tilde{d^3}} \cdots$$
\end{theorem}

In the context of uniform $\Uf$-stability, the relevant asymptotic Banach $\hGamma$-module we shall be interested is the ultraproduct $\mathcal{W}=\prod_{\omega}\uf(k_n)$, where $\uf(k_n)$ is the Lie algebra of $U(k_n)$. Note that we are only considering finite-dimensional unitary groups, so such a module is \emph{finitary}. Given a uniform asymptotic homomorphism $\phi:\hGamma \to \prod_{\omega}U(n)$ with defect $\defe(\phi) \leq_\omega \eps \in {}^*\mathbb{R}_{inf}$, this can be used to construct the asymptotic action $\pi:\hGamma \times \mathcal{W} \to \mathcal{W}$ given by $\pi(g)v = \phi(g)v\phi(g)^{-1}$, making $\mathcal{W}$ an asymptotic Banach $\hGamma$-module.

\begin{definition}
\label{def:Ulam:module}

We call such a module $\mathcal{W}$ an \emph{Ulam $\hGamma$-module supported on $\Uf$}.
\end{definition}

Also, consider the map $\alpha:\hGamma \times \hGamma \to \mathcal{W}$ given by 
\begin{equation}\label{cocycle}
    \alpha(g_1,g_2) = \frac{1}{\eps} \log(\phi(g_1)\phi(g_2)\phi(g_1g_2)^{-1}).
\end{equation}

This map $\alpha$ induces an inhomogeneous $2$-cocycle $\tilde{\alpha} : \hGamma \times \hGamma \to \tilde{\W}$, and thus defines a class in $\HH_a^2(\Gamma,\mathcal{W})$, under the usual correspondence between inhomogeneous cochains and invariant homogeneous cochains \cite[Theorem 4.2.4]{mainref}.

\begin{definition}
\label{def:Ulam:class}
The cohomology class $[\tilde{\alpha}] \in \HH_a^2(\Gamma,\mathcal{W})$ constructed in (\ref{cocycle}) is called the \emph{Ulam class} for the Ulam $\hGamma$-module $\W$.
\end{definition}
Note that any uniform asymptotic homomorphism of $\Gamma$ gives rise to an Ulam $\hGamma$-module and its own Ulam class, and we shall refer to these as \emph{Ulam classes supported on $\Uf$}.\\
Such an Ulam class vanishes, i.e. $\tilde{\alpha}$ is a coboundary, precisely when $\phi$ has the defect diminishing property. Thus Theorem \ref{thm:defdim} yields:

\begin{theorem}[{\cite[Theorem 4.2.4]{mainref}}]
\label{thm:UlamviaHa}

$\Gamma$ is uniformly $\Uf$-stable with a linear estimate with respect to $\Uf$ if and only if all Ulam classes supported on $\Uf$ vanish. In particular, if $\HH_a^2(\Gamma,\W)=0$ for every Ulam $\hGamma$-module supported on $\Uf$, then $\Gamma$ is uniformly $\Uf$-stable, with a linear estimate.
\end{theorem}

\section{Hereditary properties}
\label{s:hereditary}

In this section, we first prove Proposition \ref{intro:prop:coamenable:ac} and deduce Proposition \ref{intro:prop:coamenable} from it; then analogously we prove Proposition \ref{intro:prop:mapping:ac} and deduce Proposition \ref{intro:prop:mapping} from it. Both stability statements are not symmetric, and in fact we will see in Section \ref{s:sharp} that the converses do not hold.

\subsection{More on Zimmer-amenability}

For the proofs of Propositions \ref{intro:prop:coamenable:ac} and \ref{intro:prop:mapping:ac}, we will need a more precise version of Theorem \ref{thm:zimmer} in a special case. A regular $\Gamma$-space $S$ is said to be \emph{discrete} if it is a countable set equipped with the counting measure. It follows from the equivalent characterizations in \cite{zimmeramenable} that a discrete $\Gamma$-space is Zimmer-amenable precisely when each point stabilizer is amenable. In particular:
\begin{enumerate}
\item If $\Lambda \leq \Gamma$ is a subgroup, then the action of $\Lambda$ on $\Gamma$ by left multiplication is free, so $\Gamma$ is a discrete Zimmer-amenable $\Lambda$-space.
\item If $N \leq \Gamma$ is an amenable subgroup, then the action of $\Gamma$ on the coset space $\Gamma / N$ has stabilizers equal to conjugates of $N$, so $\Gamma / N$ is a discrete Zimmer-amenable $\Gamma$-space.
\end{enumerate}

For such spaces, we can provide an explicit chain map that implements the isomorphism in cohomology from Theorem \ref{thm:zimmer}. Indeed, the proof of Theorem \ref{thm:zimmer} works by starting with a $\Gamma$-homotopy equivalence between the two complexes:
\[
    0 \to L^\infty(\Gamma) \to L^\infty(\Gamma^2) \to L^\infty(\Gamma^3) \to \cdots \]
    \[ 0 \to L^\infty(S) \to L^\infty(S^2) \to L^\infty(S^3) \to \cdots
\]
which is then extended internally to the asymptotic version of these complexes. The case of dual asymptotic coefficients follows via some suitable identifications of the corresponding complexes (see the paragraph preceding \cite[Theorem 4.20]{mainref}). In case both $\Gamma$ and $S$ are discrete, Zimmer-amenability of $S$ boils down to amenability of point stabilizers, and the homotopy equivalence above can be chosen to be the orbit map
\[
    o_{b}^m : L^\infty(S^m) \longrightarrow L^\infty(\Gamma^m) \]
    \[o_{b}^m(f)(g_1, \ldots, g_m) = f(g_1 b, \ldots, g_m b);
\]
where $b \in S$ is some choice of basepoint; the homotopy inverse is built using invariant means on the point stabilizers \cite[Section 4.9]{frigerio}. Therefore in this case we obtain the following more explicit version of Theorem \ref{thm:zimmer}:

\begin{theorem}
\label{thm:zimmer:precise}

Let $S$ be a discrete Zimmer-amenable $\Gamma$-space, with a basepoint $b \in S$, and let $\V$ be a dual asymptotic Banach $\hGamma$-module. Then the orbit map
\[
    o_{b}^m : \LL((\hS)^m, \V) \longrightarrow \LL((\hGamma)^m, \V) \]
    \[o_{b}^m(f)(g_1, \ldots, g_m) = f(g_1 b, \ldots, g_m b);
    \]
induces induces an isomorphism between $\HH^\bullet_a(\Gamma, \V)$ and the cohomology of the complex
$$0 \xrightarrow{\tilde{d^0}} \LLtilde(\hS, \V)^{\hGamma} \xrightarrow{\tilde{d^1}} \LLtilde((\hS)^2, \V)^{\hGamma} \xrightarrow{\tilde{d^2}} \LLtilde((\hS)^3, \V)^{\hGamma} \xrightarrow{\tilde{d^3}} \cdots$$
\end{theorem}

In the two basic examples of discrete Zimmer-amenable spaces from above, we obtain:

\begin{corollary}
\label{cor:zimmer:subgroup}

Let $\Lambda \leq \Gamma$ be a subgroup, and let $\V$ be a dual asymptotic Banach $\hGamma$-module, which restricts to a dual asymptotic Banach $\hLambda$-module. Then the restriction of cochains $\LL((\hGamma)^m, \V) \to \LL((\hLambda)^m, \V)$ induces an isomorphism between $\HH^\bullet_a(\Lambda, \V)$ and the cohomology of the complex:
$$0 \xrightarrow{\tilde{d^0}} \LLtilde(\hGamma, \V)^{\hLambda} \xrightarrow{\tilde{d^1}} \LLtilde((\hGamma)^2, \V)^{\hLambda} \xrightarrow{\tilde{d^2}} \LLtilde((\hGamma)^3, \V)^{\hLambda} \xrightarrow{\tilde{d^3}} \cdots$$
\end{corollary}

\begin{proof}
Seeing $\Gamma$ as a discrete Zimmer-amenable $\Lambda$-space, with basepoint $1 \in \Gamma$, the orbit map is nothing but the restriction of cochains, and we conclude by Theorem \ref{thm:zimmer:precise}.
\end{proof}

\begin{corollary}
\label{cor:zimmer:quotient}

Let $N \leq \Gamma$ be an amenable normal subgroup, and let $\V$ be a dual asymptotic Banach ${}^{*}(\Gamma/N)$-module, which pulls back to a dual asymptotic Banach $\hGamma$-module. Then the pullback of cochains $\LL(({}^{*}(\Gamma/N))^m, \V) \to \LL((\hGamma)^m, \V)$ induces an isomorphism between $\HH^\bullet_a(\Gamma, \V)$ and $\HH^\bullet_a(\Gamma / N, \V)$.
\end{corollary}

\begin{proof}
Seeing $\Gamma/N$ as a discrete Zimmer-amenable $\Gamma$-space, with basepoint the coset $N$, the orbit map is nothing but the pullback of cochains. So Theorem \ref{thm:zimmer:precise} yields an isomorphism between $\HH^\bullet_a(\Gamma, \V)$ and the cohomology of the complex:
$$0 \xrightarrow{\tilde{d^0}} \LLtilde({}^{*}(\Gamma/N), \V)^{\hGamma} \xrightarrow{\tilde{d^1}} \LLtilde(({}^{*}(\Gamma/N))^2, \V)^{\hGamma} \xrightarrow{\tilde{d^2}} \LLtilde(({}^{*}(\Gamma/N))^3, \V)^{\hGamma} \xrightarrow{\tilde{d^3}} \cdots$$
But since the action of $\hGamma$ on both ${}^{*}(\Gamma/N)$ and $\V$ factors through ${}^{*}(\Gamma/N)$, the above complex coincides with the standard one computing $\HH^\bullet_a(\Gamma/N, \V)$.
\end{proof}

We will use these explicit isomorphisms in this section. Later, for the proof of Theorem \ref{intro:thm:lamplighters:ac}, non-discrete Zimmer-amenable spaces will also appear, but in that case we will only need the existence of an abstract isomorphism as in Theorem \ref{thm:zimmer}.

\subsection{Restrictions and coamenability}

Let $\Lambda \leq \Gamma$ be a (not necessarily coamenable) subgroup, and $\V$ be a dual asymptotic Banach $\hGamma$-module, which restricts to a dual asymptotic Banach $\hLambda$-module.
The restriction $\LLtilde((\hGamma)^\bullet, \V)^{\hGamma} \to \LLtilde((\hLambda)^\bullet, \V)^{\hLambda}$ induces a map in cohomology, called the \emph{restriction map}, and denoted
$$\res^\bullet : \Ha^\bullet(\Gamma, \V) \to \Ha^\bullet(\Lambda, \V).$$
This map behaves well with respect to Ulam classes (Definition \ref{def:Ulam:class}):

\begin{lemma}
\label{lem:restriction:Ulam}

Let $\W$ be an Ulam $\hGamma$-module supported on $\Uf$. Then $\W$ is also an Ulam $\hLambda$-module supported on $\Uf$, and the restriction map $\res^2 : \Ha^2(\Gamma, \W) \to \Ha^2(\Lambda, \W)$ sends Ulam classes to Ulam classes.
\end{lemma}

\begin{proof}
Let $\phi : \Gamma \to \Uf$ be a uniform asymptotic homomorphism, and let $\W$ be the corresponding Ulam $\hGamma$-module. Then restricting $\phi_n$ to $\Lambda$ for each $n$ yields a uniform asymptotic homomorphism $\phi|_{\Lambda} : \Lambda \to \Uf$, with $\defe(\phi|_\Lambda) \leq_{\omega} \defe(\phi)$ and endows $\W$ with an asymptotic $\hLambda$-action making it into an Ulam $\hLambda$-module supported on $\Uf$. The cocycle corresponding to $\phi$ is defined via the map
$$\alpha : \hGamma \times \hGamma \to \W : (g_1, g_2) \mapsto \frac{1}{\eps} \log(\phi(g_1)\phi(g_2)\phi(g_1 g_2)^{-1}).$$
Since $\defe(\phi|_{\Lambda}) \leq_{\omega} \eps$, restricting $\alpha$ to $\hLambda \times \hLambda$ yields a valid cocycle associated to the uniform asymptotic homomorphism $\phi_{\Lambda}$. It follows that the chain map $\LLtilde((\hGamma)^\bullet, \V)^{\hGamma} \to \LLtilde((\hLambda)^\bullet, \V)^{\hLambda}$ preserves the set of cocycles defined via uniform asymptotic homomorphisms, and therefore preserves Ulam classes.
\end{proof}

Now suppose that $\Lambda \leq \Gamma$ is coamenable. This means, by definition, that there exists a $\Gamma$-invariant mean on $\Gamma / \Lambda$; that is, there exists a linear functional $m : \ell^\infty(\Gamma/\Lambda) \to \mathbb{R}$ such that
\begin{enumerate}
    \item $m(1_{\Gamma / \Lambda}) = 1$, where $1_{\Gamma / \Lambda}$ denotes the constant function.
    \item $| m(f) | \leq  \| f \|$ for all $f \in \ell^\infty(\Gamma / \Lambda)$.
    \item $m(g \cdot f) = m(f)$ for all $g \in \Gamma$ and all $f \in \ell^\infty(\Gamma / \Lambda)$.
\end{enumerate}

As with the absolute case \cite[Lemma 3.20]{mainref}, we have the following:

\begin{lemma}
\label{lem:internalmean}

Suppose that $\Lambda \leq \Gamma$ is coamenable, and let $\V$ be a dual asymptotic Banach $\hGamma$-module. Then there exists an internal map $m : \LL(\hGamma / \hLambda, \V) \to \V$ which induces a map $\tilde{m} : \LLtilde(\hGamma / \hLambda, \V) \to \tilde{\V}$ with the following properties:
\begin{enumerate}
    \item If $\tilde{f}$ is the constant function equal to $\tilde{v} \in \Vtilde$, then $\tilde{m}(\tilde{f}) = \tilde{v}$.
    \item $\|\tilde{m}(\tilde{f})\| \leq \| \tilde{f} \|$ for all $\tilde{f} \in \LLtilde(\hGamma / \hLambda, \V)$.
    \item $\tilde{m}(g \cdot \tilde{f}) = \tilde{\pi}(g) \tilde{m}(\tilde{f})$ for all $g \in \hGamma$ and all $\tilde{f} \in \LLtilde(\hGamma / \hLambda, \V)$.
\end{enumerate}
\end{lemma}

\begin{proof}
Consider $f=\{f_n\}_{\omega} \in \LL(\hGamma / \hLambda, \V)$. Since $\V$ is a dual asymptotic ${}^*\Gamma$-module with predual $\V^{\flat}$, for each $\lambda \in \V^{\flat}$, we get an internal map 
$$f^{\lambda}:\hGamma/\hLambda \to {}^*\R : x \mapsto f(x)(\lambda).$$
Note that $f^{\lambda}$ being internal, it is of the form $\{f^{\lambda}_n\}_{\omega}$ where $f^{\lambda}_n \in \ell^{\infty}(\Gamma/\Lambda)$. This allows us to construct the internal map $m_{in}^{\lambda}:\LL(\hGamma / \hLambda, \V) \to {}^*\R$ as
$$m_{in}^{\lambda}(f)=\{m\left( f^{\lambda}_n   \right)\}_{\omega}$$
and finally $m_{in}:\LL(\hGamma / \hLambda, \V)  \to \V$ as
$$m_{in}(f)(\lambda) = m_{in}^{\lambda}(f)$$
It is straightforward to check that $m_{in}$ as defined induces a linear map $\tilde{m}:\LLtilde(\hGamma / \hLambda, \V) \to \tilde{\V}$. As for ${}^*\Gamma$-equivariance, this follows from the observation that $(g \cdot f)^{\lambda}(x)=\pi(g)f(g^{-1}x)(\lambda)$ while $(g \cdot f^{\lambda})(x)=f(g^{-1}x)(\lambda)$. The conditions on $\tilde{m}$ follow from the definition and properties of the $\Gamma$-invariant mean $m$ on $\ell^{\infty}(\Gamma/\Lambda)$. 
\end{proof}

We are now ready to prove Proposition \ref{intro:prop:coamenable:ac}. The proof goes along the lines of \cite[Proposition 8.6.2]{monod:book}.

\begin{proposition*}[Proposition \ref{intro:prop:coamenable:ac}]

Let $\Lambda \leq \Gamma$ be coamenable. Then the restriction map $\Ha^n(\Gamma, \V) \to \Ha^n(\Lambda, \V)$ is injective, for all $n \geq 0$ and all dual asymptotic Banach $\hGamma$-modules $\V$.
\end{proposition*}

\begin{proof}
We implement the asymptotic cohomology of $\Lambda$ using the complex $\LLtilde((\hGamma)^\bullet, \V)^{\hLambda}$ from Corollary \ref{cor:zimmer:subgroup}. Since the chain map that defines the restriction map factors through this complex, and the chain map $\LLtilde((\hGamma)^\bullet, \V)^{\hLambda} \to \LLtilde((\hLambda)^\bullet, \V)^{\hLambda}$ induces an isomorphism in cohomology (Corollary \ref{cor:zimmer:subgroup}), it suffices to show that the chain inclusion $\LLtilde((\hGamma)^\bullet, \V)^{\hGamma} \to \LLtilde((\hGamma)^\bullet, \V)^{\hLambda}$ induces an injective map in cohomology. Henceforth, we will refer to this as \emph{the restriction map}.

Our goal is construct a \emph{transfer} map, that is a linear map $\trans^\bullet : \Ha^\bullet(\Lambda, \V) \to \Ha^\bullet(\Gamma, \V)$ such that $\trans^\bullet \circ \res^\bullet$ is the identity on $\Ha^\bullet(\Lambda, \V)$. Then it follows at once that $\res^\bullet$ must be injective.
By the above paragraph, we may do this by constructing an internal chain map $\widetilde{\trans}^\bullet : \LLtilde((\hGamma)^\bullet, \V)^{\hLambda} \to \LLtilde((\hGamma)^\bullet, \V)^{\hGamma}$ that restricts to the identity on $\LLtilde((\hGamma)^\bullet, \V)^{\hGamma}$. \\

Let $f \in \LL((\hGamma)^k, \V)$ be such that $\tilde{f} \in \LLtilde((\hGamma)^k, \V)^{\hLambda}$. For each $x \in (\hGamma)^k$, define
$$f_x : \hGamma \to \V$$
$$f_x(g) := \pi(g)f(g^{-1}x)$$
In other words, $f_x(g)$ is just $(\rho^1(g)f)(x)$ as in (\ref{rho}). Since $\tilde{f} \in \LLtilde((\hGamma)^k, \V)^{\hLambda}$, for any $\gamma \in \hLambda$ and $g \in \hGamma$, 
$$f_x(g\gamma)-f_x(g) \in \V_{inf}$$
Let us choose representatives of left $\hLambda$-cosets in $\hGamma$ and restrict $f_x$ to this set of representatives so that we can regard $f_x$ as an internal map $f_x:\hGamma/\hLambda \to \V$. Moreover, since $f_x \in \LL(\hGamma / \hLambda, \V)$, we can apply the mean $m$ constructed in Lemma \ref{lem:internalmean} to define the internal map $\trans^k(f) : \LL((\hGamma)^k, \V) \to \LL((\hGamma)^k, \V)$ by
$$\trans^k(f)(x) = m(f_x)$$
Since $\tilde{m}$ is $\hGamma$-invariant, this means that for $g \in \hGamma$, $m(f_{gx})-\pi(g)m(f_x) \in \V_{inf}$, and implies that
$$(\trans^k(f))(gx)-\pi(g) \trans^k(f)(x) \in \V_{inf}.$$

This establishes that for $f \in \LL((\hGamma)^k, \V)$ with $\tilde{f} \in \LLtilde((\hGamma)^k, \V)^{\hLambda}$, we have $\widetilde{\trans^k(f)} \in \LLtilde((\hGamma)^k, \V)^{\hGamma}$.
Therefore $\trans^\bullet$ induces a chain map $\tilde{\trans}^\bullet : \LLtilde((\hGamma)^\bullet, \V)^{\hLambda} \to \LLtilde((\hGamma)^\bullet, \V)^{\hGamma}$. Finally, if $\tilde{f}$ is already $\hGamma$-invariant, then $f_x$ is constant up to infinitesimals, and thus $m(f_x)$ is equal, up to an infinitesimal, to the value of that constant, which is $f(x)$. This shows that $\widetilde{\trans}^k$ is the identity when restricted to $\LLtilde((\hGamma)^\bullet, \V)^{\hGamma}$, and concludes the proof.
\end{proof}

Proposition \ref{intro:prop:coamenable} is now an easy consequence.

\begin{proposition*}[Proposition \ref{intro:prop:coamenable}]

Let $\Lambda \leq \Gamma$ be coamenable. If $\Lambda$ is uniformly $\Uf$-stable with a linear estimate, then so is $\Gamma$.
\end{proposition*}

\begin{proof}
Suppose that $\Lambda$ is uniformly $\Uf$-stable with a linear estimate, and let $\Gamma$ be a coamenable supergroup of $\Lambda$. We aim to show that $\Gamma$ is also uniformly $\Uf$-stable with a linear estimate. By Theorem \ref{thm:UlamviaHa}, it suffices to show that all Ulam classes supported on $\Uf$ vanish in $\HH^2_a(\Gamma, \W)$, where $\W$ is an Ulam $\hGamma$-module. Now by Proposition \ref{intro:prop:coamenable:ac}, it suffices to show that the images of such classes under the restriction map $\res^2 : \HH^2_a(\Gamma, \W) \to \HH^2_a(\Lambda, \W)$ vanish, since the latter is injective. By Lemma \ref{lem:restriction:Ulam} these are Ulam classes of $\Lambda$. But since $\Lambda$ is uniformly $\Uf$-stable with a linear estimate, by Theorem \ref{thm:UlamviaHa} again, all Ulam classes in $\HH^2_a(\Lambda, \W)$ vanish, and we conclude.
\end{proof}

\subsection{Pullbacks and amenable kernels}

Let $N \leq \Gamma$ be an amenable normal subgroup, and let $\V$ be a dual asymptotic Banach ${}^{*}(\Gamma/N)$-module, which pulls back to a dual asymptotic Banach $\hGamma$-module. Precomposing cochains by the projection $\hGamma \to {}^{*}(\Gamma/N)$ defines the \emph{pullback} $p^\bullet : \HH^\bullet_a(\Gamma/N, \V) \to \HH^\bullet_a(\Gamma, \V)$. The following can be proven via a similar argument as in Lemma \ref{lem:restriction:Ulam}:

\begin{lemma}
\label{lem:pullback:Ulam}

Let $\W$ be an Ulam ${}^{*}(\Gamma/N)$-module. Then $\W$ is also an Ulam $\hGamma$-module, and the pullback $p^2 : \HH^2_a(\Gamma/N, \W) \to \HH^2_a(\Gamma, \W)$ sends Ulam classes to Ulam classes.
\end{lemma}

With this language, Proposition \ref{intro:prop:mapping:ac} is just a reformulation of Corollary \ref{cor:zimmer:quotient}:

\begin{proposition*}[Proposition \ref{intro:prop:mapping:ac}]

Let $N \leq \Gamma$ be an amenable normal subgroup. Then the pullback $\Ha^n(\Gamma/N, \V) \to \Ha^n(\Gamma, \V)$ is an isomorphism, for all $n \geq 0$ and all dual asymptotic Banach $\hGamma$-modules $\V$.
\end{proposition*}

And we deduce Proposition \ref{intro:prop:mapping} analogously:

\begin{proposition*}[Proposition \ref{intro:prop:mapping}]

Let $N \leq \Gamma$ be an amenable normal subgroup. If $\Gamma$ is uniformly $\Uf$-stable with a linear estimate, then so is $\Gamma / N$.
\end{proposition*}

\begin{proof}
Suppose that $\Gamma$ is uniformly $\Uf$-stable with a linear estimate, and let $N$ be an amenable normal subgroup of $\Gamma$. We aim to show that $\Gamma/N$ is also uniformly $\Uf$-stable with a linear estimate. By Theorem \ref{thm:UlamviaHa}, it suffices to show that all Ulam classes supported on $\Uf$ vanish in $\HH^2_a(\Gamma/N, \W)$, where $\W$ is an Ulam ${}^{*}(\Gamma/N)$-module. Now by Proposition \ref{intro:prop:mapping:ac}, it suffices to show that the pullback of such classes under $\HH^2_a(\Gamma/N, \W) \to \HH^2_a(\Gamma, \W)$ vanish, since the latter is an isomorphism. By Lemma \ref{lem:pullback:Ulam} these are Ulam classes of $\Gamma$. But since $\Gamma$ is uniformly $\Uf$-stable with a linear estimate, by Theorem \ref{thm:UlamviaHa} again, all Ulam classes in $\HH^2_a(\Gamma, \W)$ vanish, and we conclude.
\end{proof}

\section{Asymptotic cohomology of lamplighters}
\label{s:lamplighters}

In this section we prove Theorem \ref{intro:thm:lamplighters:ac}, which we recall for the reader's convenience:

\begin{theorem*}
Let $\Gamma, \Lambda$ be two countable groups, where $\Lambda$ is infinite and amenable. Then ${\Ha^n(\Gamma \wr \Lambda, \V) = 0}$ for all $n \geq 1$ and all finitary dual asymptotic Banach $\hGamma$-modules $\V$.
\end{theorem*}

\begin{remark}
\label{rem:separable}

In fact, the theorem will hold for a larger class of coefficients, obtained as ultraproducts of separable Banach spaces. This does not however lead to a stronger stability result: see Remark \ref{rem:sharp:lamplighters}.
\end{remark}

We start by finding a suitable Zimmer-amenable $\Gamma$-space:

\begin{lemma}[{\cite[Corollary 8, Proposition 9]{monod:lamplighters}}]
\label{lem:lamplighters:amenable}

Let $\Gamma, \Lambda$ be two countable groups, where $\Lambda$ is amenable. Let $\mu_0$ be a distribution of full support on $\Gamma$, and let $\mu$ be the product measure on $S := \Gamma^{\Lambda}$. Then $S$ is a Zimmer-amenable $(\Gamma \wr \Lambda)$-space.
\end{lemma}

The reason why this space is useful for computations is that it is \emph{highly ergodic}. Recall that a $\Gamma$-space $S$ is \emph{ergodic} if every $\Gamma$-invariant function $S \to \R$ is essentially constant. When $S$ is \emph{doubly ergodic}, that is the diagonal action of $\Gamma$ on $S \times S$ is ergodic, we even obtain ergodicity with separable coefficients, meaning that for every $\Gamma$-module $E$, every $\Gamma$-equivariant map $S \to E$ is essentially constant \cite[2.A, 4.B]{monod:lamplighters}.

\begin{lemma}[Kolmogorov {\cite[2.A, 4.B]{monod:lamplighters}}]
\label{lem:lamplighters:ergodic}

Let $\Gamma, \Lambda$ be two countable groups, where $\Lambda$ is infinite, and let $S$ be as in Lemma \ref{lem:lamplighters:amenable}. Then $S^m$ is an ergodic $(\Gamma \wr \Lambda)$-space, for every $m \geq 1$.
\end{lemma}

For our purposes, we will need an approximate version of ergodicity (namely, almost invariant functions are almost constant) and also the module $E$ will only be endowed with an approximate action of $\Gamma$. The ergodicity assumption still suffices to obtain this:

\begin{lemma}
\label{lem:approx:ergodicity}

Let $S$ be a probability measure $\Gamma$-space, and suppose that the action of $\Gamma$ on $S$ is ergodic.
Then whenever $f : S \to \mathbb{R}$ is a measurable function such that $\| g \cdot f - f \| < \eps$ for all $g \in \Gamma$, there exists a constant $c \in \mathbb{R}$ such that $|f(s) - c| < \eps$ for almost every $s \in S$.
\end{lemma}

\begin{proof}
We define $F : S \to \mathbb{R} : s \mapsto \mathrm{ess sup}_{g \in \Gamma} f(g^{-1} s)$. By construction, $F$ is $\Gamma$-invariant, and moreover $\| F - f \| < \eps$. By ergodicity, $F$ is essentially equal to a constant $c$, and thus $|f(s) - c| < \eps$ for a.e. $s \in S$.
\end{proof}

\begin{lemma}
\label{lem:approx:doubleergodicity}

Let $S$ be a probability measure $\Gamma$-space, and suppose that the action of $\Gamma$ on $S \times S$ is ergodic. Suppose moreover, that $E$ is a separable Banach space endowed with a map $\Gamma \times E \to E : v \mapsto g \cdot v$ such that $\| g \cdot v \| = \| v \|$ for all $g \in \Gamma, v \in E$.

Then whenever $f : S \to E$ is a measurable function such that $\| g \cdot f - f \| < \eps$ for all $g \in \Gamma$, where $(g \cdot f)(s) = g \cdot f(g^{-1} s)$, there exists a vector $v \in E$ such that $\|f(s) - v\| < 3 \eps$ for almost every $s \in S$.
\end{lemma}

\begin{proof}
We define $F : S \times S \to \mathbb{R} : (s, t) \mapsto \|f(s) - f(t)\|$. Then
\begin{align*}
    \| g \cdot F - F \| &= \mathrm{ess \, sup} | \, \|g \cdot f(g^{-1} s) - g \cdot f(g^{-1} t) \| - \|f(s) - f(t)\| \, | \\
    &\leq \mathrm{ess \, sup} \|g \cdot f(g^{-1} s) - g \cdot f(g^{-1} t) - (f(s) - f(t)) \| \leq 2 \| g \cdot f - f \| < 2 \eps.
\end{align*}
By the previous lemma, there exists a constant $c$ such that $|F(s, t) - c| < 2 \eps$ for all $\eps > 0$. If $c < \eps$, then $|f(s) - f(t)| < 3 \eps$ for a.e. $s, t \in S$, which implies the statement.

Otherwise, $\|f(s) - f(t)\| > \eps$ for a.e. $s, t \in S$. Let $D \subset E$ be a countable dense subset. Then for each $d \in D$ the set $f^{-1}(B_{\eps / 2}(d))$ is a measurable subset of $S$, and the union of such sets covers $S$. Since $D$ is countable, there must exist $d \in D$ such that $f^{-1}(B_{\eps / 2}(d))$ has positive measure. But for all $s, t$ in this set, $\|f(s) - f(t)\| < \eps$, a contradiction.
\end{proof}

\begin{remark}
\label{rem:finitary}

This lemma really strongly relies on the separability assumption, and it is the reason why Theorem \ref{intro:thm:lamplighters:ac} includes the finitary assumption (see Remark \ref{rem:sharp:lamplighters}).
\end{remark}

We thus obtain:

\begin{proposition}
\label{prop:doubleergodicity:constant}

Let $S$ be a doubly ergodic $\Gamma$-space. Let $(V_n)_{n \geq 1}$ be a sequence of separable dual Banach spaces such that $\V = \prod\limits_\omega V_n$ has the structure of a dual asymptotic Banach $\Gamma$-module be the corresponding asymptotic $\hGamma$-module. Then the natural inclusion $\Vtilde^{\hGamma} \to \LLtilde(\hS, \V)^{\hGamma}$ is an isomorphism.
\end{proposition}

\begin{proof}
Let $f \in \LLb(\hS, \V) = \prod\limits_\omega L^\infty(S, V_n)$ be a lift of an element $\tilde{f} \in \LLtilde(\hS, \V)^{\hGamma}$. We write $f = (f_n)_\omega$. Then fact that $\tilde{f}$ is $\hGamma$-invariant means that for every sequence $(g_n)_{n \in \N} \subset \Gamma$ it holds $(g_n \cdot f_n - f_n)_{\omega} \in \LLinf(\hS, \V)$. Since this holds for \emph{every} sequence $(g_n)_{n \in \N}$, a diagonal argument implies that there exists $\eps \in \hRinf$ such that for every $g \in \Gamma$ it holds $(g \cdot f_n - f_n)_\omega \in \hRinf$. It then follows from Lemma \ref{lem:approx:doubleergodicity} that there exist $(v_n)_\omega \in \V$ such that $(f_n - 1_{v_n}) \in \LLinf(\hS, \V)$. Therefore $f$ represents the same element of $\LLtilde(\hS, \V)$ as the image of an element of $\V$. Since $\tilde{f}$ is $\hGamma$-invariant, the corresponding element is actually in $\Vtilde^{\hGamma}$.
\end{proof}

We are finally ready to prove Theorem \ref{intro:thm:lamplighters:ac}:

\begin{proof}[Proof of Theorem \ref{intro:thm:lamplighters:ac}]

Let $\Gamma, \Lambda$ be countable groups, where $\Lambda$ is infinite and amenable. By Lemma \ref{lem:lamplighters:amenable}, using the same notation, $S$ is a Zimmer-amenable $(\Gamma \wr \Lambda)$-space. Therefore we can apply Theorem \ref{thm:zimmer}, and obtain that the following complex computes $\HH^*_a(\Gamma \wr \Lambda; \V)$:
$$0 \xrightarrow{\tilde{d^0}} \LLtilde(\hS, \V)^{\hGamma} \xrightarrow{\tilde{d^1}} \LLtilde((\hS)^2, \V)^{\hGamma} \xrightarrow{\tilde{d^2}} \LLtilde((\hS)^3, \V)^{\hGamma} \xrightarrow{\tilde{d^3}} \cdots$$
Now by Lemma \ref{lem:lamplighters:ergodic}, $S^m$ is a doubly ergodic $(\Gamma \wr \Lambda)$-space, for every $m \geq 1$. Thus Proposition \ref{prop:doubleergodicity:constant} applies, and the natural inclusion $\Vtilde^{\hGamma} \to \LLtilde((\hS)^m, \V)^{\hGamma}$ is an isomorphism for every $m \geq 1$. Thus the above complex is isomorphic to
$$0 \xrightarrow{\tilde{d^0}} \Vtilde^{\hGamma} \xrightarrow{\tilde{d^1}} \Vtilde^{\hGamma} \xrightarrow{\tilde{d^2}} \Vtilde^{\hGamma} \xrightarrow{\tilde{d^3}} \cdots$$
Each differential $\tilde{d^m}$ is an alternating sum of $(m+1)$ terms all equal to each other. Therefore $\tilde{d^m}$ is the identity whenever $m$ is even, and it vanishes whenever $m$ is odd. The conclusion follows.
\end{proof}

\section{Thompson groups}
\label{s:thompson}

In this section we prove Theorem \ref{intro:thm:F}. The statement for $F'$ will be a special case of a more general result for a large family of \emph{self-similar} groups. The most general statement is the following:

\begin{theorem}
\label{thm:selfsimilar}

Let $\Gamma$ be a group, $\Gamma_0$ a subgroup with the following properties:
\begin{enumerate}
    \item There exists $g \in \Gamma$ such that the groups $\{ g^i \Gamma_0 g^{-i} : i \in \mathbb{Z} \}$ pairwise commute;
    \item Every finite subset of $\Gamma$ is contained in some conjugate of $\Gamma_0$.
\end{enumerate}
Then $\HH^n_a(\Gamma, \V) = 0$ for all $n \geq 1$ and all finitary dual asymptotic Banach $\hGamma$-modules $\V$. In particular, $\Gamma$ is uniformly $\Uf$-stable, with a linear estimate.
\end{theorem}

The theorem applies to the following large family of groups of homeomorphisms of the real line:

\begin{corollary}
\label{cor:bsupp}

Let $\Gamma$ be a proximal, boundedly supported group of orientation-preserving homeomorphisms of the line. Then $\HH^n_a(\Gamma, \V) = 0$ for all $n \geq 1$ and all finitary dual asymptotic Banach $\hGamma$-modules $\V$. In particular, $\Gamma$ is uniformly $\Uf$-stable, with a linear estimate.
\end{corollary}

\begin{remark}
The fact that such groups have no quasimorphisms is well-known: see e.g. \cite{usimple, cc, monod:lamplighters}.
\end{remark}

We refer the reader to Section \ref{s:PLPP} for the relevant definitions. In Corollary \ref{cor:F} we will apply Corollary \ref{cor:bsupp} to Thompson's group $F'$; the result for Thompson's group $F$ will follow from Proposition \ref{intro:prop:coamenable}. We deduce the stability of Thompson's group $T$ and $V$ from these general criteria in Section \ref{s:TV}.

\subsection{Self-similar groups}

In this section we prove Theorem \ref{thm:selfsimilar}. This will be done in a series of lemmas:

\begin{lemma}
\label{lem:metabelian:kernel}

Let $\Gamma$ be a group, and suppose that there exists $g \in \Gamma$ and $\Gamma_0 \leq \Gamma$ such that $\{g^i \Gamma_0 g^{-i} : i \in \mathbb{Z} \}$ pairwise commute. Then there exists an epimorphism $\Gamma_0 \wr \mathbb{Z} \to \langle \Gamma_0, g \rangle$ with amenable (in fact, metabelian) kernel.
\end{lemma}

This is well-known and stated without proof in \cite{monod:lamplighters}. We include a proof for completeness.

\begin{proof}
To make a clear distinction, we denote by $H$ the abstract group $\Gamma_0$, and by $\Gamma_0$ the subgroup of $\Gamma$. So we want to construct an epimorphism $H \wr \mathbb{Z} \to \langle \Gamma_0, g \rangle \leq \Gamma$ with metabelian kernel. We define naturally
$$\varphi((g_i)_{i \in \mathbb{Z}}, p) = \left( \prod_{i \in \mathbb{Z}} t^i g_i t^{-i} \right) t^p.$$
Note that this product is well-defined since there are only finitely many non-identity terms, and the order does not matter since different conjugates commute. By construction $\varphi$ is injective on $H_i$, that is the copy of $H$ supported on the $i$-th coordinate in $H \wr \mathbb{Z}$. Let $K := \ker \varphi \cap \bigoplus_i H_i$, and note that $K$ is the kernel of the retraction $H \wr \mathbb{Z} \to \mathbb{Z}$ restricted to $\ker \varphi$. So it suffices to show that $K$ is abelian.

Let $g, h \in K$ and write them as $(g_i)_{i \in \mathbb{Z}}$ and $(h_i)_{i \in \mathbb{Z}}$ (we omit the $\mathbb{Z}$-coordinate since it is always $0$). We need to show that $g$ and $h$ commute. We have
$$1_\Gamma = \varphi(g) = \prod\limits_{i \in \mathbb{Z}} t^i g_i t^{-i} \quad \text{ and thus } \quad g_0 = \prod \limits_{i \neq 0} t^i g_i t^{-i} \in \Gamma.$$
But now $g_0$ belongs to a group generated by conjugates of $\Gamma_0$ in $\Gamma$ that commute with it. In particular this implies that $g_0$ and $h_0$ commute in $\Gamma$. Since $\varphi|_{H_0}$ is injective, this shows that $g_0$ and $h_0$ commute in $H_0$. Running the same argument on the other coordinates, we obtain that $g_i$ and $h_i$ commute in $H_i$, for all $i \in \mathbb{Z}$, and thus $g$ and $h$ commute.
\end{proof}

The next facts are all contained in the literature:

\begin{lemma}[{\cite[Proposition 10]{monod:lamplighters}}]
\label{lem:coamenable:commuting}

Suppose that $\Gamma_0 \leq \Gamma$ is such that every finite subset of $\Gamma$ is contained in some $\Gamma$-conjugate of $\Gamma_0$. Then $\Gamma_0$ is coamenable in $\Gamma$.
\end{lemma}

\begin{lemma}[\cite{coamenable}]
\label{lem:coamenable:constructions}

Let $K \leq H \leq \Gamma$.
\begin{enumerate}
    \item If $K$ is coamenable in $\Gamma$, then $H$ is coamenable in $\Gamma$;
    \item If $K$ is coamenable in $H$ and $H$ is coamenable in $\Gamma$, then $K$ is coamenable in $\Gamma$.
\end{enumerate}
\end{lemma}

\begin{remark}
We warn the reader that if $K$ is coamenable in $\Gamma$, then $K$ need not be coamenable in $H$ \cite{coamenable}.
\end{remark}

We are now ready to prove Theorem \ref{thm:selfsimilar}:

\begin{proof}[Proof of Theorem \ref{thm:selfsimilar}]

Let $\Gamma, \Gamma_0$ and $g$ be as in the statement. By Lemma \ref{lem:metabelian:kernel}, there exists a map $\Gamma_0 \wr \mathbb{Z} \to \langle \Gamma_0, g \rangle$ with metabelian kernel. By Theorem \ref{intro:thm:lamplighters:ac} and Proposition \ref{intro:prop:mapping:ac}, we have $\HH^n_a(\langle \Gamma_0, g \rangle, \V)$ for all $n \geq 1$ and all finitary dual asymptotic Banach $\hGamma$-modules $\V$. Now by Lemma \ref{lem:coamenable:commuting}, $\Gamma_0$ is coamenable in $\Gamma$. Finally, by Lemma \ref{lem:coamenable:constructions}, $\langle \Gamma_0, g \rangle$ is coameanble in $\Gamma$. Proposition \ref{intro:prop:coamenable:ac} allows to conclude.
\end{proof}

\subsection{Groups of homeomorphisms of the line}
\label{s:PLPP}

Let $\Gamma$ be a group acting by homeomorphisms on the real line. We say that the action is \emph{proximal} if for all reals $a < b$ and $c < d$ there exists $g \in \Gamma$ such that $g \cdot a < c < d < g \cdot b$. The \emph{support} of $g \in \Gamma$ is the set $\{ x \in \mathbb{R} : g \cdot x \neq x \}$. We say that $\Gamma$ is \emph{boundedly supported} if every element has bounded support. Note that boundedly supported homeomorphisms are automatically orientation-preserving.

\begin{proof}[Proof of Corollary \ref{cor:bsupp}]
Let $\Gamma$ be as in the statement. Let $\Gamma_0$ be the subgroup of elements whose support is contained in $[0, 1]$. Let $g \in \Gamma$ be such that $g(0) > 1$: such an element exists because the action of $\Gamma$ is proximal. Then it follows by induction, and the fact that $\Gamma$ is orientation-preserving, that the intervals $\{ g^i[0, 1] : i \in \mathbb{Z} \}$ are pairwise disjoint. Therefore the conjugates $g^i \Gamma_0 g^{-i}$ pairwise commute.

Since $\Gamma$ is boundedly supported, for every finite subset $A \subset \Gamma$ there exists $n$ such that the support of each element of $A$ is contained in $[-n, n]$. By proximality, there exists $h \in \Gamma$ such that $h(0) < -n$ and $h(1) > n$. Then $h \Gamma_0 h^{-1}$ is the subgroup of elements whose support is contained in $[-n, n]$, in particular it contains $A$.

Thus Theorem \ref{thm:selfsimilar} applies and we conclude.
\end{proof}

Let us now show how to obtain the statements on $F$ and $F'$ from Theorem \ref{intro:thm:F} from Corollary \ref{cor:bsupp} and Proposition \ref{intro:prop:coamenable}. We refer the reader to \cite{CFP} for more details on Thompson's groups.

Thompson's group $F$ is the group of orientation-preserving piecewise linear homeomorphisms of the interval, with breakpoints in $\mathbb{Z}[1/2]$ and slopes in $2^\mathbb{Z}$. The derived subgroup $F'$ coincides with the subgroup of boundedly supported elements. The action of $F'$ (and thus $F$) on $[0, 1]$ preserves $\mathbb{Z}[1/2] \cap (0, 1)$, and acts highly transitively on it; that is, for every pair of ordered $n$-tuples in $\mathbb{Z}[1/2] \cap (0, 1)$ there exists an element of $F'$ sending one to the other.

\begin{corollary}
\label{cor:F}

Thompson's groups $F$ and $F'$ are uniformly $\Uf$-stable, with a linear estimate.
\end{corollary}

\begin{proof}
We identify $(0, 1)$ with the real line. The group $F'$ is boundedly supported, and it is proximal, since it acts transitively on ordered pairs of a dense set. Therefore Corollary \ref{cor:bsupp} applies and $F'$ is uniformly $\Uf$-stable, with a linear estimate.

Since the quotient $F/F'$ is abelian, thus amenable, we see that $F'$ is coamenable in $F$, and thus conclude from Proposition \ref{intro:prop:coamenable} that $F'$ is uniformly $\Uf$-stable, with a linear estimate.
\end{proof}

\begin{remark}
We could also deduce the stability of $F$ from the stability of $F'$ more directly, without appealing to Proposition \ref{intro:prop:coamenable}. Indeed, since $F'$ is uniformly $\Uf$-stable, simple, and not linear, every homomorphism $F' \to U(n)$ is trivial - something we will come back to in the next section. Therefore uniform $\Uf$-stability of $F'$ implies that every uniform asymptotic homomorphism $F' \to \Uf$ is uniformly close to the trivial one. It follows that every uniform asymptotic homomorphism $F \to \Uf$ is uniformly asymptotically close to one that factors through $\mathbb{Z}^2$. We conclude by the stability of amenable groups \cite{kazhdan, mainref}.
\end{remark}

Other groups to which these criteria apply include more piecewise linear groups \cite{PL}, such as the Stein--Thompson groups \cite{stein}, or the golden ratio Thompson group of Cleary \cite{Ftau1, Ftau2}. In such generality some more care is needed, since the commutator subgroup is sometimes a proper subgroup of the boundedly supported subgroup. The criteria also apply for the piecewise proejective groups of Monod \cite{PP1} and Lodha--Moore \cite{PP3}. In this case, further care is needed, since the role of the commutator subgroup in the proofs above has to be taken by the double commutator subgroup \cite{PP2}. This ties back to Question \ref{q} from the introduction.

\subsection{$T$ and $V$}
\label{s:TV}

In this section, we show how our previous results allow to prove stability of groups of homeomorphisms of the circle and of the Cantor set as well. For simplicity of the exposition, we only focus on Thompson's groups $T$ and $V$, but the proofs generalize to some analogously defined groups, with the appropriate modifications.
Our proof will involve a bounded generation argument for stability that was pioneered in \cite{BOT}. We will only use it a simple version thereof, closer to the one from \cite{BC}. Recall that $\Gamma$ is said to be \emph{boundedly generated} by the collection of subgroups $\mathcal{H}$ if there exists $k \geq 1$ such that the sets $\{ H_1 \cdots H_k : H_i \in \mathcal{H} \}$ cover $\Gamma$.

\begin{lemma}
\label{lem:bg}

Let $\Gamma$ be a discrete group. Suppose that there exists a subgroup $H \leq \Gamma$ with the following properties:
\begin{enumerate}
    \item Every homomorphism $H \to U(n)$ is trivial;
    \item $H$ is uniformly $\Uf$-stable (with a linear estimate);
    \item $\Gamma$ is boundedly generated by the conjugates of $H$.
\end{enumerate}
Then $\Gamma$ is uniformly $\Uf$-stable (with a linear estimate).
\end{lemma}

\begin{proof}
Let $\phi_n : \Gamma \to U(d_n)$ be a uniform asymptotic homomorphism with $\defe(\phi_n) =: \eps_n$. Then $\phi_n |_H : H \to U(d_n)$ is a uniform asymptotic homomorphism of $H$, therefore it is $\delta_n$-close to a homomorphism, where $\delta_n \to 0$. But by assumption such a homomorphism must be trivial, so $\| \phi_n(h) - I_{k_n} \| \leq \delta_n$ for all $n$. The same holds for all conjugates of $H$, up to replacing $\delta_n$ by $\delta_n + 2 \eps_n$.

By bounded generation, there exists $k \geq 1$ such that each $g \in \Gamma$ can be written as $g = h_1 \cdots h_k$, where each $h_i$ belongs to a conjugate of $H$. We estimate:
\begin{align*}
    \| \phi_n(g) - I_{d_n} \| &= \left\| \phi_n\left( \prod\limits_{i = 1}^k h_i \right) - I_{d_n} \right\| \leq \left\| \phi_n\left( \prod\limits_{i = 1}^{k-1} h_i \right) \phi_n(h_k) - I_{d_n} \right\| + \eps_n \\
    &= \left\| \phi_n\left( \prod\limits_{i = 1}^{k-1} h_i \right) - I_{d_n} \right\| + \| \phi_n(h_k) - I_{d_n} \| + \eps_n \leq \cdots \\
    & \cdots \leq \sum\limits_{i = 1}^k \| \phi_n(h_i) - I_{d_n} \| + k \eps_n \leq k(\delta_n + \eps_n).
\end{align*}
Therefore $\phi_n$ is $k(\delta_n + \eps_n)$-close to the trivial homomorphism, and we conclude.
\end{proof}

Thompson's group $T$ is the group of orientation-preserving piecewise linear homeomorphisms of the circle $\mathbb{R}/\mathbb{Z}$ preserving $\mathbb{Z}[1/2]/\mathbb{Z}$, with breakpoints in $\mathbb{Z}[1/2]/\mathbb{Z}$, and slopes in $2^{\mathbb{Z}}$. Given $x \in \mathbb{Z}[1/2]/\mathbb{Z}$, the stabilizer of $x$ is naturally isomorphic to $F$. Moreover, the germ stabilizer $T(x)$ (i.e. the group consisting of elements that fix pointwise some neighbourhood of $x$) is isomorphic to $F'$.

\begin{corollary}
\label{cor:T}

Thompson's group $T$ is uniformly $\Uf$-stable with a linear estimate.
\end{corollary}

\begin{proof}
We claim that Lemma \ref{lem:bg} applies with $H = T(0) \cong F'$. Item $1.$ follows from the fact $F'$ does not embed into $U(n)$ (for instance because it contains $F$ as a subgroup, which is finitely generated and not residually finite, and so cannot be linear by Mal'cev's Theorem \cite{malcev}), and $F'$ is simple \cite{CFP}. Also, $F'$ is uniformly $\Uf$-stable with a linear estimate, by Corollary \ref{cor:F}. Therefore we are left to show the bounded generation statement. We will show that for every $g \in T$ there exist $x, y \in \mathbb{Z}[1/2]/\mathbb{Z}$ such that $g \in T(x) T(y)$. This suffices because $T$ acts transitively on $\mathbb{Z}[1/2]/\mathbb{Z}$, so $T(x)$ and $T(y)$ are both conjugate to $H = T(0)$.

Let $1 \neq g \in T$, and choose $x \neq y \in \mathbb{Z}[1/2]/\mathbb{Z}$ such that $g(y) \notin \{ x, y \}$. Let $I$ be a small dyadic arc around $y$ such that $x \notin \overline{I}$ and $x, y \notin \overline{g(I)}$. Choose an element $f \in T(x)$ such that $f(I) = g(I)$. Let $h$ be an element supported on $I$ such that $h|_I = f^{-1}g|_I$. Since $x \notin \overline{I}$, we also have $h \in T(x)$. Moreover $h^{-1}f^{-1}g|_I = \mathrm{id}|_I$, so $h^{-1}f^{-1}g \in G(y)$. We conclude that $g = f h \cdot h^{-1} f^{-1} g \in T(x) T(y)$.
\end{proof}

Thompson's group $V$ can be described as a group of homeomorphisms of the dyadic Cantor set $X := 2^{\mathbb{N}}$. A \emph{dyadic brick} is a clopen subset of the form $X_{\sigma} := \sigma \times 2^{\mathbb{N}_{> k}}$, for some $\sigma \in 2^k$, and every two dyadic bricks are canonically homeomorphic via $X_{\sigma} \to X_{\tau} : \sigma \times x \mapsto \tau \times x$. An element $g \in V$ is defined by two finite partitions of $V$ of the same size into dyadic bricks, that are sent to each other via canonical homeomorphisms.

\begin{corollary}
\label{cor:V}

Thompson's group $V$ is uniformly $\Uf$-stable, with a linear estimate.
\end{corollary}

The proof is very similar to the proof for $T$, so we only sketch it:

\begin{proof}[Sketch of proof]
Let $x \in 2^{\mathbb{N}}$ be a dyadic point, that is a sequence that is eventually all $0$, and let $V(x)$ denote the subgroup of $V$ consisting of elements that fix a neighbourhood of $x$ pointwise. The same argument as in the proof of Corollary \ref{cor:T} shows that $V$ is boundedly generated by conjugates of $V(x)$.

Now $V(x)$ is isomorphic to a directed union of copies of $V$, which is finitely generated and simple \cite{CFP}, so by Mal'cev's Theorem every homomorphism $V(x) \to U(n)$ is trivial. Finally, $V(x)$ contains a copy $V_0$ of $V$ such that the pair $(V(x), V_0)$ satisfies the hypotheses of Theorem \ref{thm:selfsimilar} (see \cite[Proposition 4.3.4]{konstantin} and its proof). We conclude by Lemma \ref{lem:bg}.
\end{proof}

\section{Sharpness of our results}
\label{s:sharp}

In this section we point out certain ways in which our results are sharp, by providing explicit counterexamples to generalizations and converses.

\begin{remark}
\label{rem:sharp:lamplighters}

There is a notion of \emph{strong Ulam stability}, where one takes $\Uf$ to include unitary groups of infinite-dimensional Hilbert spaces as well, typically equipped with the operator norm. It is shown in \cite{BOT} that a subgroup of a strongly Ulam stable group is Ulam stable. Therefore it is clear that Theorem \ref{intro:thm:lamplighters} does not hold for strong Ulam stability. Even restricting to separable Hilbert spaces does not help: it follows from the construction in \cite{BOT} that if a \emph{countable} group contains a free subgroup, then separable Hilbert spaces already witness the failure of strong Ulam stability. 

The framework of stability via asymptotic cohomology can be developed in this general setting as well, with dual asymptotic Banach modules that are not finitary. Therefore the counterexample above shows that Theorem \ref{intro:thm:lamplighters:ac} really needs the finitary assumption (see Remark \ref{rem:finitary}). The fact that we could obtain dual asymptotic Banach modules obtained as ultraproducts of separable spaces, analogously to \cite{monod:lamplighters}, does not help, since the dual asymptotic Banach modules arising from a stability problem over infinite-dimensional Hilbert spaces are not of this form, even when the Hilbert space are separable.
\end{remark}

\begin{remark}
We proved in Proposition \ref{intro:prop:coamenable} that if $\Lambda$ is coamenable in $\Gamma$ and $\Lambda$ is uniformly $\Uf$-stable with a linear estimate, then so is $\Gamma$. The converse does not hold. Let $F_n$ be a free group of rank $n \geq 2$. Then $\Lambda := \bigoplus_{n \geq 1} F_n$ admits a non-trivial quasimorphism, so it is not uniformly $U(1)$-stable \cite{BOT}, in particular it is not uniformly $\Uf$-stable. However, $\Lambda$ is coamenable in $F_n \wr \mathbb{Z}$, which is uniformly $\Uf$-stable with a linear estimate by Theorem \ref{intro:thm:lamplighters}.

On the other hand, if we replace ``coamenable'' by ``finite index'', then the converse does hold. This follows from the induction procedure in \cite{BOT} for Ulam stability, as detailed in \cite[Lemma II.22]{gamm}; the same proof can be generalized to all submultiplicative norms \cite[Lemma 4.3.6]{mainref}.
\end{remark}

\begin{remark}
We proved in Proposition \ref{intro:prop:mapping} that if $N$ is an amenable normal subgroup of $\Gamma$, and $\Gamma$ is uniformly $\Uf$-stable with a linear estimate, then so is $\Gamma/N$. The converse does not hold. Let $\Gamma$ be the lift of Thompson's group $T$, that is, the group of orientation-preserving homeomorphisms of $\mathbb{R}$ that commute with the group $\mathbb{Z}$ of integer translations and induce $T$ on the quotient $\mathbb{R}/\mathbb{Z}$. These groups fit into a central extension
$$1 \to \mathbb{Z} \to \Gamma \to T \to 1.$$
Now $T$ is uniformly $\Uf$-stable with a linear estimate, by Corollary \ref{cor:T}, however $\Gamma$ is not: it is not even uniformly $U(1)$-stable, by \cite{BOT}, since it has a non-trivial quasimorphism \cite{surungroupe}.
\end{remark}

The next two remarks show that some results from \cite{mainref} are also sharp.

\begin{remark}
The fundamental result of \cite{mainref} is that the vanishing of asymptotic cohomology implies uniform $\Uf$-stability. The converse does not hold. Indeed, since $\uf(1) \cong \mathbb{R}$ with trivial adjoint action (because $U(1)$ is abelian), it follows that the implication of Theorem \ref{thm:UlamviaHa} specializes to: If $\HH^2_a(\Gamma, \hR) = 0$, then $\Gamma$ is uniformly $U(1)$-stable, where $\hR$ is seen as a dual asymptotic $\hGamma$-module with a trivial $\hGamma$ action.

Now, let again $\Gamma$ be the lift of Thompson's group $T$, so that $\Gamma$ contains a central subgroup $\mathbb{Z}$ with $\Gamma / \mathbb{Z} \cong T$. The fact that $\Gamma$ is not uniformly $U(1)$-stable implies that $\HH^2_a(\Gamma, \hR) \neq 0$. But Proposition \ref{intro:prop:mapping:ac} then shows that $\HH^2_a(T, \hR) \neq 0$ either. However, $T$ is uniformly $\Uf$-stable with a linear estimate, by Corollary \ref{cor:T}. Morally, this is due to the fact that $\HH^2_b(\Gamma, \R) \cong \HH^2_b(T, \R) \cong \R$, but the former is spanned by a quasimorphisms, while the latter is not (see e.g. \cite[Chapter 5]{scl}).
\end{remark}

\begin{remark}
In \cite{BOT} it is shown that groups admitting non-trivial quasimorphisms are not uniformly $U(1)$-stable. In \cite[Proposition 1.0.6]{mainref} this result is sharpened: the authors show that $\Gamma$ is uniformly $U(1)$-stable if and only if the non-zero element in the image of $\HH^2_b(\Gamma, \mathbb{Z})$ in $\HH^2_b(\Gamma, \mathbb{R})$ have Gromov norm $\| \cdot \|$ bounded away from $0$. They use this to show that a finitely presented group is uniformly $U(1)$-stable if and only if it admits no non-trivial quasimorphism \cite[Corollary 1.0.10]{mainref}.

The hypothesis of finite presentability is necessary. Let $\Gamma_n$ denote the lift of Thompson's group $T$ to $\mathbb{R}/n\mathbb{Z}$. That is, $\Gamma_n$ is the group of orientation-preserving homeomorphisms of the topological circle $\mathbb{R}/n\mathbb{Z}$, which commute with the cyclic group of rotations $\mathbb{Z}/n\mathbb{Z}$ and induce $T$ on the quotient $\mathbb{R}/\mathbb{Z}$. Now $T$ has no unbounded quasimorphisms (see e.g. \cite[Chapter 5]{scl}), and so $\Gamma_n$ also has no unbounded quasimorphisms (this follows from the left exactness of the quasimorphism functor \cite[Remark 2.90]{scl}). Therefore the group $\Gamma := \bigoplus_{n \geq 2} \Gamma_n$ has no unbounded quasimorphisms.

However, we claim that $\Gamma$ is not uniformly $U(1)$-stable. By \cite[Proposition 1.0.6]{mainref}, it suffices to show that there exist bounded cohomology classes $0 \neq \rho_n \in \im(\HH^2_b(\Gamma, \mathbb{Z}) \to \HH^2_b(\Gamma, \mathbb{R})$ such that $\| \rho_n \| \to 0$. We let $\rho_n$ be the Euler class of the representation $\Gamma \to \Gamma_n \to \Homeo^+(\mathbb{R}/n\mathbb{Z})$, which admits an integral representative and so lies in the image of $\HH^2_b(\Gamma, \mathbb{Z})$ (see \cite{ghys} for more information about Euler classes of circle actions). Moreover, using the terminology of \cite{extensioncriterion}, the representation is minimal, unbounded, and has a centralizer of order $n$. Therefore $\| \rho_n \| = 1/2n$ by \cite[Corollary 1.6]{extensioncriterion}, and we conclude.

Note that $\Gamma$ is countable but infinitely generated. It would be interesting to produce a finitely generated example (which would necessarily be infinitely presented).
\end{remark}

\section{Approximation properties}
\label{s:last}

In this section we discuss open problems about approximation properties of the groups treated in this paper, and their relation to our results. We recall the following notions:

\begin{definition}
Let $\mathcal{G}$ be a family of metric groups. We say that $\Gamma$ is (pointwise, uniformly) \emph{$\mathcal{G}$-approximable} if there exists a (pointwise, uniform) asymptotic homomorphism $\phi_n : \Gamma \to G_n \in \mathcal{G}$ that is moreover \emph{asymptotically injective}, meaning that for all $g \in \Gamma, g \neq 1$ it holds $\liminf\limits_{n \to \infty} \phi_n(g) > 0$.
\end{definition}

The above terminology is not standard: most of the literature only deals with the pointwise notion, and refers to that as \emph{$\mathcal{G}$-approximability}. The notion of uniform approximability appeared in \cite{ultrametric} with the name of \emph{strong $\mathcal{G}$-approximability}.

\begin{example}
If $\mathcal{G}$ is the family of symmetric groups equipped with the normalized Hamming distance, then pointwise $\mathcal{G}$-approximable groups are called \emph{sofic} \cite{sofic1, sofic2}.

If $\mathcal{G}$ is the family of unitary groups equipped with the Hilbert--Schmidt distance, then pointwise $\mathcal{G}$-approximable groups are called \emph{hyperlinear} \cite{hyperlinear}.

All amenable and residually finite groups are sofic, and all sofic groups are hyperlinear. It is a major open question to determine whether there exists a non-sofic group.
\end{example}

In our context of submultiplicative norms on unitary groups, the following two notions of approximability have been studied:

\begin{example}
Let $\mathcal{G}$ be the family of unitary groups equipped with the operator norm. Then pointwise $\mathcal{G}$-approximable groups are called MF \cite{MF}. All amenable groups are MF \cite{amenableMF}. It is an open problem to determine whether there exists a non-MF group.

Let $\mathcal{G}$ be the family of unitary groups equipped with the Frobenius norm, or more generally with a Schatten $p$-norm, for $1 < p < \infty$. Groups that are not pointwise $\mathcal{G}$-approximable have been constructed in \cite{DGLT, pnorm}. This is one of the very few cases in which a non-example for pointwise approximability is known.
\end{example}

The following observation is well-known, and due to Glebsky and Rivera \cite{glebskyrivera} and Arzhantseva and P\u{a}unescu in the pointwise symmetric case \cite{arzhantsevapaunescu}. We give a general proof for reference:

\begin{proposition}
\label{prop:ap}

Let $\mathcal{G}$ be a family of metric groups that are locally residually finite, and let $\Gamma$ be a finitely generated group. Suppose that $\Gamma$ is (pointwise, uniformly) $\mathcal{G}$-stable and (pointwise, uniformly) $\mathcal{G}$-approximable. Then $\Gamma$ is residually finite.
\end{proposition}

The hypothesis on $\mathcal{G}$ covers all cases above. When the groups in $\mathcal{G}$ are finite, this is clear, and when they are linear, this follows from Mal'cev's Theorem \cite{malcev}.

\begin{proof}
We proceed with the proof without specifying the type of asymptotic homomorphisms, closeness, and approximability: the reader should read everything as pointwise, or everything as uniform.

Let $\phi : \Gamma \to \mathcal{G}$ be an asymptotically injective asymptotic homomorphism. By stability, there exists a sequence of homomorphisms $\psi : \Gamma \to \mathcal{G}$ which is asymptotically close to $\phi$. Since $\phi$ is asymptotically injective, for each $g \in \Gamma$ there exists $N$ such that $\phi_n(g) \geq \rho$ for all $n \geq N$ and some $\rho = \rho(g) > 0$. Up to taking a larger $N$, we also have that $\psi_n(g) \geq \rho/2$, in particular $\psi_n(g) \neq 1$. Since $\psi_n(\Gamma)$ is a finitely generated subgroup of $G_n \in \mathcal{G}$, it is residually finite by hypothesis, and so $\psi_n(g)$ survives in some finite quotient of $\psi_n(\Gamma)$. Since this is also a finite quotient of $\Gamma$, we conclude that $\Gamma$ is residually finite.
\end{proof}

In the special case of pointwise stability and Thompson's group $F$, we obtain the following more general version of a remark of Arzhantseva and Paunescu \cite[Open problem]{arzhantsevapaunescu}:

\begin{corollary}
\label{cor:F:pwapprox}

Let $\mathcal{G}$ be the family of symmetric groups with the normalized Hamming distance, the family of unitary groups with the Hilbert--Schmidt norm, or the family of unitary groups with the operatorn norm. If Thompson's group $F$ is pointwise $\mathcal{G}$-stable, then it is not pointwise $\mathcal{G}$-approximable, and in particular it is non-amenable.
\end{corollary}

As we mentioned in the introduction, the amenability of Thompson's group $F$ is one of the most outstanding open problems in modern group theory.

\begin{proof}
Thompson's group $F$ is not residually finite \cite{CFP}. So it follows from Proposition \ref{prop:ap} that it cannot be simultaneously pointwise $\mathcal{G}$-stable and pointwise $\mathcal{G}$-approximable. The last statement follows from the fact that amenable groups are sofic, hyperlinear, and MF.
\end{proof}

On the other hand, our results allow to settle the uniform approximability of Thompson's groups, with respect to unitary groups and submultiplicative norms:

\begin{corollary}
\label{cor:F:uniapprox}

As usual, let $\Uf$ be the family of unitary groups equipped with submultiplicative norms. Then Thompson's groups $F, F', T$ and $V$ are not uniformly $\Uf$-approximable. The same holds for $\Gamma \wr \Lambda$, whenever $\Lambda$ is infinite and amenable, and $\Gamma$ is non-abelian.
\end{corollary}

We remark that Thompson's groups $T$ and $V$ are generally regarded as good candidates for counterexamples to approximability problems.

\begin{proof}
The statement for $F, T$ and $V$ follows from Theorem \ref{intro:thm:F} and Proposition \ref{prop:ap}, together with the fact that they are not residually finite, and the statement for $F'$ (which is not finitely generated) follows from the fact that $F'$ contains a copy of $F$ \cite{CFP}. The lamplighter case follows from Theorem \ref{intro:thm:lamplighters} and Proposition \ref{prop:ap}, together with the fact that such lamplighters are not residually finite \cite{grunberg}.
\end{proof}

We do not know whether Thompson's groups are uniformly $\mathcal{G}$-approximable, when $\mathcal{G}$ is the family of unitary groups equipped with the Hilbert--Schmidt norm, and we conjecture that this is not the case. In the next section, we examine the case of symmetric groups via a more direct argument.

\subsection{Approximations by symmetric groups}

We end by proving, by a cohomology-free argument, that some of the groups studied in this paper are not uniformly approximable by symmetric groups, in a strong sense. For the rest of this section, we denote by $\mathcal{S}$ the family of symmetric groups equipped with the normalized Hamming distance. Our main result is an analogue of Corollary \ref{cor:bsupp} for this approximating family (see Section \ref{s:PLPP} for the relevant definitions):

\begin{proposition}
\label{prop:sym}

Let $\Gamma$ be a proximal, boundedly supported group of orientation-preserving homeomorphisms of the line. Then every uniform asymptotic homomorphism $\phi_n : \Gamma' \to S_{k_n} \in \mathcal{S}$ is uniformly asymptotically close to the trivial one. In particular, $\Gamma'$ is uniformly $\mathcal{S}$-stable, and not uniformly $\mathcal{S}$-approximable.
\end{proposition}

The non-approximability follows from the fact that $\Gamma'$ is non-trivial (see Lemma \ref{lem:fixpoints}). Note that for $\Gamma$ as in the statement, $\Gamma'$ is simple \cite[Theorem 1.1]{usimple}, so in particular every homomorphism $\Gamma' \to S_{k_n}$ is trivial.

The proof relies on known results on the flexible uniform stability of amenable groups \cite{BC} and uniform perfection of groups with proximal actions \cite{usimple}. The finiteness of the groups in $\mathcal{S}$ will play a crucial role. We start with the following lemma:

\begin{lemma}
\label{lem:fixpoints}

Let $\Gamma$ be as in Proposition \ref{prop:sym}. Then $\Gamma'$ is non-trivial, and the action of $\Gamma'$ on the line has no global fixpoints.
\end{lemma}

\begin{proof}
If $\Gamma'$ is trivial, then $\Gamma$ is abelian. This contradicts that the action is proximal and boundedly supported. Indeed, given $g \in \Gamma$, since $g$ is centralized, the action of $\Gamma$ on $\mathbb{R}$ must preserve the support of $g$, which is a proper subset of $\mathbb{R}$. But then the action cannot be proximal.

Now the set of global fixpoints of $\Gamma'$ is a closed subset $X \subset \mathbb{R}$. Since $\Gamma'$ is normal in $\Gamma$, the action of $\Gamma$ preserves $X$. But the action of $\Gamma$ on $\mathbb{R}$ is proximal, in particular every orbit is dense, and since $X$ is closed we obtain $X = \mathbb{R}$. That is, $\Gamma'$ acts trivially on $\mathbb{R}$. Since $\Gamma$ is a subgroup of $\Homeo^+(\mathbb{R})$, this implies that $\Gamma'$ is trivial, which contradicts the previous paragraph.
\end{proof}

We proceed with the proof:

\begin{proof}[Proof of Proposition \ref{prop:sym}]
It follows from \cite[Theorem 1.1]{usimple} that $\Gamma'$ is $2$-uniformly perfect; that is, every element of $\Gamma'$ may be written as the product of at most $2$ commutators (this uses the proximality hypothesis). Therefore it suffices to show that there exists a constant $C$ such that for all $g, h \in \Gamma'$ it holds $d_{k_n}(\phi_n([g, h]), \id_{k_n}) \leq C\eps_n$, where $d_{k_n}$ denotes the Hamming distance on $S_{k_n}$ and $\eps_n := \defe(\phi_n)$. We drop the subscript $n$ on $\phi$ and $\eps$ for clarity.

Now let $g, h \in \Gamma'$, and let $I, J \subset \mathbb{R}$ be bounded intervals such that $g$ is supported on $I$ and $h$ is supported on $J$. Since $\Gamma'$ acts without global fixpoints by Lemma \ref{lem:fixpoints}, there exists $t \in \Gamma'$ such that $t \cdot \inf(J) > \sup(I)$. Since $\Gamma'$ is orientation-preserving, the same holds for all powers of $t$. In particular $[g, t^i h t^{-i}] = 1$ for all $i \geq 1$. Next, we apply \cite[Theorem 1.2]{BC} to the amenable group $\langle t \rangle$, to obtain an integer $N$ such that $k_n \leq N \leq (1 + 1218 \eps)k_n$ and a permutation $\tau$ in $S_N$ such that $d_N(\phi(t)^i, \tau^i) \leq 2039 \eps$ for all $i \in \mathbb{Z}$. Here $d_N$ denotes the normalized Hamming distance on the symmetric group $S_N$, and $\phi$ is extended to a map $\phi : \Gamma' \to S_N$ with every $\phi(g)$ fixing each point in $\{ k_n+1, \ldots, N \}$. We compute (using $\tau^{N!} = \id_N)$:
\begin{align*}
    d_{k_n}(\phi([g, h]), \id_{k_n}) &\leq d_N(\phi([g, h]), \id_N) \leq d_N([\phi(g), \phi(h)], \id_N) + O(\eps) \\
    &= d_N([\phi(g), \tau^{N!}\phi(h)\tau^{-N!}], \id_N) + O(\eps) \\
    &\leq d_N([\phi(g), \phi(t^{N!})\phi(h)\phi(t^{-N!})], \id_N) + O(\eps) \\
    &\leq d_N(\phi([g, t^{N!} h t^{-N!}]), \id_N) + O(\eps) \\
    &= d_N(\phi(1), \id_N) + O(\eps) \leq O(\eps).
\end{align*}
Thus, there exists a constant $C$ independent of $g$ and $h$ ($C = 20000$ suffices) such that $d_{k_n}(\phi([g, h]), \id_{k_n}) \leq C\eps$, which concludes the proof.
\end{proof}

\begin{corollary}
Consider the Thompson groups $F', F, T$.
\begin{enumerate}
    \item Every asymptotic homomorphism $\phi_n : F' \to S_{k_n} \in \mathcal{S}$ is uniformly asymptotically close to the trivial one.
    \item Every asymptotic homomorphism $\phi_n : F \to S_{k_n} \in \mathcal{S}$ is uniformly asymptotically close to one that factors through the abelianization.
    \item Every asymptotic homomorphism $\phi_n : T \to S_{k_n} \in \mathcal{S}$ is uniformly asymptotically close to the trivial one.
\end{enumerate}
\end{corollary}

\begin{proof}
Item $1.$ is an instance of Proposition \ref{prop:sym}: indeed $F'$ satisfies the hypotheses for $\Gamma$, and $F'' = F'$ since $F'$ is simple.
For Item $2.$, pick a section $\sigma : \Ab(F) \to F$, and define $\psi_n(g) := \phi_n(\sigma(\Ab(g)))$. Using that $\psi_n|_{F'}$ is uniformly asymptotically close to the sequence of trivial maps, we obtain that $\phi_n$ and $\psi_n$ are uniformly asymptotically close, and $\psi_n$ factors as $F \to \Ab(F) \xrightarrow{\phi_n \circ \sigma} S_{k_n}$.
Finally, Item $3.$ follows again from Item $1.$ and the fact that every element of $T$ can be written as a product of two elements in isomorphic copies of $F'$ (see the proof of Corollary \ref{cor:T}).
\end{proof}

The corollary immediately implies that $F, F'$ and $T$ are not uniformly $\mathcal{S}$-approximable, and that $F'$ and $T$ are uniformly $\mathcal{S}$-stable. Since $F$ has infinite abelianization, it follows from \cite[Theorem 1.4]{BC} that it is not uniformly $\mathcal{S}$-stable. However the corollary together with \cite[Theorem 1.2]{BC} implies that it is \emph{flexibly} uniformly $\mathcal{S}$-stable; that is, every uniform asymptotic homomorphism is uniformly close to a sequence of homomorphisms taking values in a symmetric group of slightly larger degree. The case of Thompson's group $V$ can also be treated analogously (see the sketch of proof of Corollary \ref{cor:V}).

\footnotesize

\bibliographystyle{alpha}
\bibliography{ref}

\begin{thebibliography}{DCGLT20}

\bibitem[AEG94]{zimmeramenable}
S.~Adams, G.~A. Elliott, and T.~Giordano.
\newblock Amenable actions of groups.
\newblock {\em Trans. Amer. Math. Soc.}, 344(2):803--822, 1994.

\bibitem[Alp20]{alpeev}
A.~Alpeev.
\newblock Lamplighters over non-amenable groups are not strongly {U}lam stable.
\newblock {\em arXiv preprint arXiv:2009.11738}, 2020.

\bibitem[And22]{konstantin}
K.~Andritsch.
\newblock Bounded cohomology of groups acting on {C}antor sets.
\newblock {\em arXiv preprint arXiv:2210.00459}, 2022.

\bibitem[AP15]{arzhantsevapaunescu}
G.~Arzhantseva and L.~P\u{a}unescu.
\newblock Almost commuting permutations are near commuting permutations.
\newblock {\em J. Funct. Anal.}, 269(3):745--757, 2015.

\bibitem[BC20]{BC}
O.~Becker and M.~Chapman.
\newblock Stability of approximate group actions: uniform and probabilistic.
\newblock {\em J. Eur. Math. Soc. (JEMS)}, To appear. \textit{arXiv:
  2005.06652}, 2020.

\bibitem[BLR18]{PP2}
J.~Burillo, Y.~Lodha, and L.~Reeves.
\newblock Commutators in groups of piecewise projective homeomorphisms.
\newblock {\em Adv. Math.}, 332:34--56, 2018.

\bibitem[BM99]{lattices}
M.~Burger and N.~Monod.
\newblock Bounded cohomology of lattices in higher rank {L}ie groups.
\newblock {\em J. Eur. Math. Soc. (JEMS)}, 1(2):199--235, 1999.

\bibitem[BM02]{rigidity}
M.~Burger and N.~Monod.
\newblock Continuous bounded cohomology and applications to rigidity theory.
\newblock {\em Geom. Funct. Anal.}, 12(2):219--280, 2002.

\bibitem[BNR21]{Ftau2}
J.~Burillo, B.~Nucinkis, and L.~Reeves.
\newblock An irrational-slope {T}hompson's group.
\newblock {\em Publ. Mat.}, 65(2):809--839, 2021.

\bibitem[BOT13]{BOT}
M.~Burger, N.~Ozawa, and A.~Thom.
\newblock On {U}lam stability.
\newblock {\em Israel J. Math.}, 193(1):109--129, 2013.

\bibitem[BS16]{PL}
R.~Bieri and R.~Strebel.
\newblock {\em On groups of {PL}-homeomorphisms of the real line}, volume 215
  of {\em Mathematical Surveys and Monographs}.
\newblock American Mathematical Society, Providence, RI, 2016.

\bibitem[Bur11]{extensioncriterion}
M.~Burger.
\newblock An extension criterion for lattice actions on the circle.
\newblock In {\em Geometry, rigidity, and group actions}, Chicago Lectures in
  Math., pages 3--31. Univ. Chicago Press, Chicago, IL, 2011.

\bibitem[Cal09]{scl}
D.~Calegari.
\newblock {\em scl}, volume~20 of {\em MSJ Memoirs}.
\newblock Mathematical Society of Japan, Tokyo, 2009.

\bibitem[Cal10]{laws}
D.~Calegari.
\newblock Quasimorphisms and laws.
\newblock {\em Algebr. Geom. Topol.}, 10(1):215--217, 2010.

\bibitem[CDE13]{MF}
J.~R. Carri\'{o}n, M.~Dadarlat, and C.~Eckhardt.
\newblock On groups with quasidiagonal {$C^*$}-algebras.
\newblock {\em J. Funct. Anal.}, 265(1):135--152, 2013.

\bibitem[CFP96]{CFP}
J.~W. Cannon, W.~J. Floyd, and W.~R. Parry.
\newblock Introductory notes on {R}ichard {T}hompson's groups.
\newblock {\em Enseign. Math. (2)}, 42(3-4):215--256, 1996.

\bibitem[Cle00]{Ftau1}
S.~Cleary.
\newblock Regular subdivision in {$\bold Z[\frac{1+\sqrt 5}{2}]$}.
\newblock {\em Illinois J. Math.}, 44(3):453--464, 2000.

\bibitem[DCGLT20]{DGLT}
M.~De~Chiffre, L.~Glebsky, A.~Lubotzky, and A.~Thom.
\newblock Stability, cohomology vanishing, and nonapproximable groups.
\newblock {\em Forum Math. Sigma}, 8:Paper No. e18, 37, 2020.

\bibitem[FF21]{ultrametric}
F.~Fournier-Facio.
\newblock Ultrametric analogues of {U}lam stability of groups.
\newblock {\em arXiv preprint arXiv:2105.00516}, 2021.

\bibitem[FFL23]{cc}
F.~Fournier-Facio and Y.~Lodha.
\newblock Second bounded cohomology of groups acting on 1-manifolds and
  applications to spectrum problems.
\newblock {\em Adv. Math.}, 428:Paper No. 109162, 2023.

\bibitem[FFLM21]{binate}
F.~Fournier-Facio, C.~L{\"o}h, and M.~Moraschini.
\newblock Bounded cohomology and binate groups.
\newblock {\em J. Aust. Math. Soc.}, To appear. \textit{arXiv:2111.04305},
  2021.

\bibitem[Fri17]{frigerio}
R.~Frigerio.
\newblock {\em Bounded cohomology of discrete groups}, volume 227 of {\em
  Mathematical Surveys and Monographs}.
\newblock American Mathematical Society, Providence, RI, 2017.

\bibitem[Gam11]{gamm}
C.~Gamm.
\newblock $\varepsilon$-representations of groups and {U}lam stability.
\newblock Master's thesis, Georg-August-Universit{\"a}t G{\"o}ttingen, 2011.

\bibitem[GG17]{usimple}
S.~R. Gal and J.~Gismatullin.
\newblock Uniform simplicity of groups with proximal action.
\newblock {\em Trans. Amer. Math. Soc. Ser. B}, 4:110--130, 2017.
\newblock With an appendix by N. Lazarovich.

\bibitem[Ghy01]{ghys}
E.~Ghys.
\newblock Groups acting on the circle.
\newblock {\em Enseign. Math. (2)}, 47(3-4):329--407, 2001.

\bibitem[GLMR23]{mainref}
L.~Glebsky, A.~Lubotzky, N.~Monod, and B.~Rangarajan.
\newblock Asymptotic cohomology and uniform stability for lattices in
  semisimple groups.
\newblock {\em arXiv preprint arXiv:2301.00476}, 2023.

\bibitem[GR09]{glebskyrivera}
L.~Glebsky and L.~M. Rivera.
\newblock Almost solutions of equations in permutations.
\newblock {\em Taiwanese J. Math.}, 13(2A):493--500, 2009.

\bibitem[Gro82]{gromov}
M.~Gromov.
\newblock Volume and bounded cohomology.
\newblock {\em Inst. Hautes \'{E}tudes Sci. Publ. Math.}, (56):5--99 (1983),
  1982.

\bibitem[Gro99]{sofic1}
M.~Gromov.
\newblock Endomorphisms of symbolic algebraic varieties.
\newblock {\em J. Eur. Math. Soc. (JEMS)}, 1(2):109--197, 1999.

\bibitem[Gru57]{grunberg}
K.~W. Gruenberg.
\newblock Residual properties of infinite soluble groups.
\newblock {\em Proc. London Math. Soc. (3)}, 7:29--62, 1957.

\bibitem[GS87]{surungroupe}
E.~Ghys and V.~Sergiescu.
\newblock Sur un groupe remarquable de diff\'{e}omorphismes du cercle.
\newblock {\em Comment. Math. Helv.}, 62(2):185--239, 1987.

\bibitem[HNN49]{embedding}
G.~Higman, B.~H. Neumann, and H.~Neumann.
\newblock Embedding theorems for groups.
\newblock {\em J. London Math. Soc.}, 24:247--254, 1949.

\bibitem[Iva85]{ivanov}
N.~V. Ivanov.
\newblock Foundations of the theory of bounded cohomology.
\newblock volume 143, pages 69--109, 177--178. 1985.
\newblock Studies in topology, V.

\bibitem[Joh72]{johnson}
B.~E. Johnson.
\newblock {\em Cohomology in {B}anach algebras}.
\newblock Memoirs of the American Mathematical Society, No. 127. American
  Mathematical Society, Providence, R.I., 1972.

\bibitem[Kaz82]{kazhdan}
D.~Kazhdan.
\newblock On {$\varepsilon $}-representations.
\newblock {\em Israel J. Math.}, 43(4):315--323, 1982.

\bibitem[LM16]{PP3}
Y.~Lodha and J.~T. Moore.
\newblock A nonamenable finitely presented group of piecewise projective
  homeomorphisms.
\newblock {\em Groups Geom. Dyn.}, 10(1):177--200, 2016.

\bibitem[LO20]{pnorm}
A.~Lubotzky and I.~Oppenheim.
\newblock Non {$p$}-norm approximated groups.
\newblock {\em J. Anal. Math.}, 141(1):305--321, 2020.

\bibitem[Mal40]{malcev}
A.~Malcev.
\newblock On isomorphic matrix representations of infinite groups.
\newblock {\em Rec. Math. [Mat. Sbornik] N.S.}, 8 (50):405--422, 1940.

\bibitem[Man05]{manning}
J.~F. Manning.
\newblock Geometry of pseudocharacters.
\newblock {\em Geom. Topol.}, 9:1147--1185, 2005.

\bibitem[MN23]{monodnariman}
N.~Monod and S.~Nariman.
\newblock Bounded and unbounded cohomology of homeomorphism and diffeomorphism
  groups.
\newblock {\em Invent. Math.}, 232(3):1439--1475, 2023.

\bibitem[Mon01]{monod:book}
N.~Monod.
\newblock {\em Continuous bounded cohomology of locally compact groups}, volume
  1758 of {\em Lecture Notes in Mathematics}.
\newblock Springer-Verlag, Berlin, 2001.

\bibitem[Mon13]{PP1}
N.~Monod.
\newblock Groups of piecewise projective homeomorphisms.
\newblock {\em Proc. Natl. Acad. Sci. USA}, 110(12):4524--4527, 2013.

\bibitem[Mon22]{monod:lamplighters}
N.~Monod.
\newblock Lamplighters and the bounded cohomology of {T}hompson's group.
\newblock {\em Geom. Funct. Anal.}, 32(3):662--675, 2022.

\bibitem[MP03]{coamenable}
N.~Monod and S.~Popa.
\newblock On co-amenability for groups and von {N}eumann algebras.
\newblock {\em C. R. Math. Acad. Sci. Soc. R. Can.}, 25(3):82--87, 2003.

\bibitem[MS04]{monodshalom}
N.~Monod and Y.~Shalom.
\newblock Cocycle superrigidity and bounded cohomology for negatively curved
  spaces.
\newblock {\em J. Differential Geom.}, 67(3):395--455, 2004.

\bibitem[Neu37]{BH}
B.~H. Neumann.
\newblock Some remarks on infinite groups.
\newblock {\em J. Lond. Math. Soc.}, 1(2):120--127, 1937.

\bibitem[R\u08]{hyperlinear}
F.~R\u{a}dulescu.
\newblock The von {N}eumann algebra of the non-residually finite {B}aumslag
  group {$\langle a,b|ab^3a^{-1}=b^2\rangle$} embeds into {$R^\omega$}.
\newblock In {\em Hot topics in operator theory}, volume~9 of {\em Theta Ser.
  Adv. Math.}, pages 173--185. Theta, Bucharest, 2008.

\bibitem[Ste92]{stein}
M.~Stein.
\newblock Groups of piecewise linear homeomorphisms.
\newblock {\em Trans. Amer. Math. Soc.}, 332(2):477--514, 1992.

\bibitem[Tur38]{turing}
A.~M. Turing.
\newblock Finite approximations to {L}ie groups.
\newblock {\em Ann. of Math. (2)}, 39(1):105--111, 1938.

\bibitem[TWW17]{amenableMF}
A.~Tikuisis, S.~White, and W.~Winter.
\newblock Quasidiagonality of nuclear {$C^\ast$}-algebras.
\newblock {\em Ann. of Math. (2)}, 185(1):229--284, 2017.

\bibitem[Ula60]{ulam}
S.~M. Ulam.
\newblock {\em A collection of mathematical problems}.
\newblock Interscience Tracts in Pure and Applied Mathematics, no. 8.
  Interscience Publishers, New York-London, 1960.

\bibitem[vN29]{vN}
J.~von Neumann.
\newblock Beweis des {E}rgodensatzes und des {H}-{T}heorems in der neuen
  {M}echanik.
\newblock {\em Z. Phys.}, 57(1):30--70, 1929.

\bibitem[Wei00]{sofic2}
B.~Weiss.
\newblock Sofic groups and dynamical systems.
\newblock {\em Sankhy\={a} Ser. A}, 62(3):350--359, 2000.
\newblock Ergodic theory and harmonic analysis (Mumbai, 1999).

\end{thebibliography}

\normalsize

\vspace{0.5cm}

\noindent{\textsc{Department of Mathematics, ETH Z\"urich, Switzerland}}

\noindent{\textit{E-mail address:} \texttt{francesco.fournier@math.ethz.ch}} \\

\noindent{\textsc{Einstein Institute of Mathematics, Hebrew University of Jerusalem, Israel}}

\noindent{\textit{E-mail address:} \texttt{bharatrm.rangarajan@mail.huji.ac.il}}

\end{document}